\documentclass[11pt]{amsart}
\usepackage{color,amssymb}%,refcheck}%,showlabels}
\usepackage{relsize}
\usepackage[all]{xy}
\usepackage[colorlinks]{hyperref}
\usepackage{marvosym}

\newcommand{\E}{\mathbf{E}}
\newcommand{\D}{\mathbf{D}}
\newcommand{\B}{\mathcal B}

\newcommand{\w}{\omega}
\newcommand{\cl}{\operatorname{cl}}
\newcommand{\Int}{\operatorname{int}}

\newcommand{\N}{\mathbb N}

\newcommand{\F}{\mathcal F}

\newcommand{\R}{\mathbb{R}}

\newcommand{\PR}{\operatorname{PR}}

\newcommand*{\defeq}{\stackrel{\mathsmaller{\mathsf{def}}}{=}}

\DeclareSymbolFont{extraup}{U}{zavm}{m}{n}
\DeclareMathSymbol{\varheart}{\mathalpha}{extraup}{86}
\DeclareMathSymbol{\vardiamond}{\mathalpha}{extraup}{87}

\newtheorem{theorem}{Theorem}[section]
\newtheorem{proposition}[theorem]{Proposition}

\newtheorem{corollary}[theorem]{Corollary}
\newtheorem{question}[theorem]{Question}
\newtheorem{lemma}[theorem]{Lemma}

\newtheorem{claim}[theorem]{Claim}
\newtheorem{problem}[theorem]{Problem}

\theoremstyle{definition}

\title[Countably compact extensions and cardinal characteristics of the continuum]{Countably compact extensions and cardinal characteristics of the continuum}
\author[S. Bardyla]{Serhii Bardyla}%$^{\ast}$
\author[P. Nyikos]{Peter Nyikos}
\thanks{\Cross \hspace{-0,5mm}  Sadly, Peter Nyikos passed away in February 2024, when this work was in the final stage of the preparation.
Therefore the rest of the authors decided to finish the research and to submit the paper with his name as coauthor. 
This is our tribute to our dear friend Peter}
\author[L.~Zdomskyy]{Lyubomyr Zdomskyy}

\address{Faculty of Mathematics, University of Vienna, Austria.}
\email{sbardyla@gmail.com}
\urladdr{http://www.logic.univie.ac.at/$\tilde{\ }$bardylas55/}

\address{Department of Mathematics, University of South Carolina, USA.}
\email{NYIKOS@math.sc.edu}
\urladdr{https://people.math.sc.edu/nyikos/}

\address{Institute of Discrete Mathematics and Geometry,
TU Wien, Austria.}
\email{lzdomsky@gmail.com}
\urladdr{https://dmg.tuwien.ac.at/zdomskyy/}

\thanks{The research of the first named author was funded in whole by the Austrian Science Fund FWF [10.55776/ESP399]. The research of the third named author was funded in whole by the Austrian Science Fund FWF [10.55776/I5930]}



\begin{document}
\begin{abstract}
In this paper, we show that the existence of certain  first-countable compact-like extensions is equivalent to the equality between corresponding cardinal characteristics of the continuum. For instance, $\mathfrak b=\mathfrak s=\mathfrak c$ if and only if every regular first-countable space of weight $< \mathfrak c$ can be densely embedded into a regular first-countable countably compact space. 
\end{abstract}
\maketitle

\section{Introduction}
The main motivation for investigation of regular first-countable countably compact spaces is the following problem of Nyikos which is listed among 20 central problems in Set-theoretic Topology by Hru\v{s}ak and Moore~\cite{HM}.      

%%%%%%%%%%%%%%%%%%%%%%%%%%%%%%%%%%%%%%%%%%%%%%%%%%%%%%%%%%%%%%%%%%%%%%%%%%%%%%%%%%%%%%%%%%%%%%%%%%%%%%%%%%%%%%%

\begin{problem}[Nyikos]\label{pN}
Does ZFC imply the existence of a regular separable first-countable countably compact non-compact space?
\end{problem}

Following~\cite{BZ}, a regular separable first-countable countably compact space is called a {\em Nyikos} space.
Consistent examples of noncompact Nyikos spaces were constructed by Franklin and Rajagopalan~\cite{FR} (under $\w_1=\mathfrak t$) and by Ostaszewski~\cite{Ost} (under $\mathfrak b=\mathfrak c$).
Weiss~\cite{Weiss} showed that under MA every perfectly normal Nyikos space is compact. Proper Forcing Axiom or, simply, PFA is a stronger version of Martin's Axiom. Basic information about PFA can be found in~\cite{Moore}. For fruitful applications of PFA in Topology see~\cite{PFA1,PFA2,PFA3,PFA, Viale}.  The following result was proved by Nyikos and Zdomskyy~\cite{NZ}.

\begin{theorem}\label{NZ}
{\em (PFA)} Every normal Nyikos space is compact.    
\end{theorem}

%So, non-normal Nyikos spaces are of particular interest. 

A space $X$ is called {\em $\mathbb R$-rigid} if each continuous real-valued function on $X$ is constant. For more about rigid spaces see \cite{BO,5,6,7,Tz}.
Another problem related to Nyikos spaces appeared in~\cite[Problem C65]{Prob}.

\begin{problem}[Tzannes]\label{Tzannes}
Does there exist a regular (separable, first countable) countably compact $\mathbb R$-rigid space?    
\end{problem}

Taking into account the properties in brackets, Tzannes problem can be considered as an ultimate version of Nyikos problem, as $\mathbb R$-rigid spaces, being not Tychonoff, are not locally compact.

Bardyla and Osipov~\cite{BO} constructed a ZFC example of a regular  countably compact $\mathbb R$-rigid space. Bardyla and Zdomskyy~\cite{BZ} obtained the following answer to Problem~\ref{Tzannes} under the consistent assumption ($\varheart$): ``$\w_1=\mathfrak t<\mathfrak b=\mathfrak c$ and there exists a $P_{\mathfrak c}$-point in $\beta(\w)$''.

\begin{theorem}[Bardyla, Zdomskyy]\label{BZZ}
$(\varheart)$ There exists an $\mathbb R$-rigid Nyikos space.    
\end{theorem}

The proof of Theorem~\ref{BZZ} had two steps. During the first one we constructed a regular separable first-countable $\mathbb R$-rigid space $X$ of cardinality $<\mathfrak c$. The second step was to embed densely the space $X$ into a first-countable countably compact space $Y$. Note that $Y$ will be automatically an $\mathbb R$-rigid Nyikos space.  This was done using the following result proved in~\cite
%[Theorem 5.7]
{BZ} under the assumption ($\heartsuit$): ``$\mathfrak b=\mathfrak c$ and there exists a $P_{\mathfrak c}$-point in $\beta(\w)$''. 

\begin{theorem}[Bardyla, Zdomskyy]\label{BZ}
$(\heartsuit)$ Each regular first-countable space of cardinality $< \mathfrak c$ embeds densely into a regular first-countable countably compact space. 
\end{theorem}

Similar problems concerning embedding of topological spaces into first-countable compact-like spaces are known in General Topology. For instance, the most relevant to this paper is the following one posed in~\cite{Ste}.

\begin{problem}[Stephenson]\label{Ste}
Does every locally feebly compact first-countable regular space embed densely into a feebly compact first-countable regular space?   
\end{problem}

Problem~\ref{Ste} was solved affirmatively by Simon and Tironi~\cite{ST}.
%\begin{theorem}[Simon, Tironi]\label{st}
%Each regular first-countable locally feebly compact space embeds densely into a regular first-countable feebly compact space.
%\end{theorem}
In case of Tychonoff spaces the following result was obtained in~\cite{TT}.

\begin{theorem}[Terada, Terasawa]\label{tt}
Each Tychonoff first-countable locally pseudocompact space embeds densely into a Tychonoff first-countable pseudocompact space.
\end{theorem}

%Countably compact extensions were also investigated in~\cite{BBR,BO}

The mentioned results motivate the following general question which we address in this paper.

\begin{question}\label{mainq}
When does a space $X$ embed into a regular first-countable compact-like space?
\end{question}

%In this paper we show that within the class of spaces of weight $<\mathfrak c$ the affirmative answer to Question~\ref{mainq} is equivalent to an equality between some well-known cardinal characteristics of the continuum. 
%Also, we develop a universal method of destroying countable discrete subsets of a space $X$ with small weight.

\section{Preliminaries}\label{prelim}
By $\mathbb N$ we denote the set of positive integers, i.e. $\mathbb N=\w\setminus \{0\}$. The cardinality of the continuum is denoted by $\mathfrak c$. The set of all infinite subsets of a countable set $A$ is denoted by $[A]^{\w}$. The set of all finite subsets of a set $A$ is denoted by $[A]^{<\w}$. A family $\mathcal A\subset[\w]^\w$ is called {\em almost disjoint} if $|A\cap B|<\w$ for each $A,B\in\mathcal A$. A family $\mathcal S\subseteq [\w]^{\w}$ is called {\em splitting} if for any $A\in [\w]^\w$ there exists $S\in\mathcal S$ which splits $A$, i.e.  $|A\cap S|=|A\cap (\w\setminus S)|=\w$. For two subsets $A$ and $B$ of $\w$ we write $A\subseteq^* B$ or $B\supseteq^* A$ if $|A\setminus B|<\w$. A family $\mathcal T\subset [\w]^\w$ is called a {\em maximal tower} if it is well-ordered by the relation $\supseteq^*$ and there is no set $A\in[\w]^\w$ such that $A\subseteq^* T$ for all $T\in\mathcal T$. For any functions $f,g\in\w^\w$ we write $f\leq^*g$ if $f(n)\leq g(n)$ for all but finitely many $n\in\w$. 
A family $\Phi\subseteq \w^\w$ is called {\em unbounded} if there exists no $g\in\w^\w$ such that $f\leq^* g$ for all $f\in\Phi$. We write $f\equiv c$ if $f$ is a constant function with value $c$.

We shall use the following cardinal characteristics of the continuum:

\begin{itemize}
    \item $\mathfrak s =\min\{|\mathcal S|: \mathcal S\subset [\w]^{\w}$ is a splitting family$\}$;
    \item $\mathfrak t =\min\{|\mathcal T|: \mathcal T\subset [\w]^{\w}$ is a maximal tower$\}$;
    \item $\mathfrak b =\min\{|\Phi|: \Phi\subset \w^{\w}$ is an  unbounded family$\}$.
\end{itemize}

It is well-known that $\w_1\leq \mathfrak t\leq \min\{\mathfrak s,\mathfrak b\}$.

An ultrafilter $\F$ on $\w$ is called a {$P_{\mathfrak c}$ point} if $\F$ has a base which forms a maximal tower of cardinality $\mathfrak c$. The existence of such an ultrafilter is consistent with ZFC and, in particular, follows from $\mathfrak t=\mathfrak c$. 

Let $X,Y$ be topological spaces. A map $f:X\rightarrow Y$ is called an {\em embedding of $X$ into $Y$} if the map $f': X\rightarrow f(X)$, obtained by restricting the range of $f$, is a homeomorphism. The {\em weight} of a topological space $X$ is the smallest cardinal $w(X)$ such that $X$ has a base of cardinality $w(X)$.

A space $X$ is called 
\begin{itemize}
    \item {\em countably compact} if each closed discrete subset of $X$ is finite;
    \item {\em pseudocompact} if $X$ is Tychonoff and each continuous real-valued function is bounded;
    \item {\em feebly compact} if every locally finite family of open subsets of $X$ is finite.
   % \item {\em separable} if $X$ contains a dense countable set.
  %  \item {\em first-countable} if each $x\in X$ possesses a countable open neighborhood base.  
\end{itemize}
 It is known that countable compactness implies feeble compactness. In the case of Tychonoff spaces pseudocompactness is equivalent to feeble compactness.

The {\em Pixley–Roy hyperspace} $\PR(X)$ of a space $X$ is the set $[X]^{<\w}$ of all finite subsets of $X$ endowed with the topology generated by the base consisting of the sets $$[F,U]=\{A\in [X]^{<\w}: F\subseteq A\subseteq U\},$$ where $F\in [X]^{<\w}$ and $U$ is open in $X$.  It can be easily checked that $\PR(X)$ is Hausdorff and zero-dimensional, if $X$ is Hausdorff. Moreover, if $X$ is first-countable, then so is $\PR(X)$.

The notions used but not defined in this paper are standard and can be found in~\cite{Buk,int, Eng, Kun}.
 
\section{Main results}
The majority of the results in this section are equivalences, each being based on an ``embedding theorem'' combined with a ``non-embedding theorem'', which are proved in sections 4 and 5, respectively.

\begin{theorem}\label{comp}
The following assertions are equivalent:
\begin{enumerate}
    \item $\w_1=\mathfrak c$.
    \item Every first-countable Tychonoff space of weight $<\mathfrak c$ embeds into a Hausdorff first-countable compact space.
    \item Each separable first-countable locally compact normal space of cardinality $<\mathfrak c$ embeds into a Hausdorff first-countable compact space.
\end{enumerate}
\end{theorem}

\begin{theorem} \label{lc} The following assertions are equivalent.   
\begin{enumerate}	
\item $\mathfrak b = \mathfrak c$.
\item Every Hausdorff locally compact first-countable space of weight $< \mathfrak c$ can be densely embedded in a Hausdorff first-countable locally compact countably compact space.
 \item Every Hausdorff locally compact first-countable space of cardinality $< \mathfrak c$ can be densely embedded in a Hausdorff first-countable  countably compact space.
\end{enumerate}
  \end{theorem}

Observe that $\mathfrak b=\mathfrak s=\mathfrak c$ follows from ($\heartsuit$). So, the following result generalizes Theorem~\ref{BZ}.
  
\begin{theorem} \label{sbc} The following assertions are equivalent.   
\begin{enumerate}	
\item $\mathfrak b = \mathfrak s= \mathfrak c$.

%\item Every Hausdorff zero-dimensional first-countable space of weight $< \mathfrak c$ can be densely embedded in a Hausdorff zero-dimensional first-countable countably compact space.
%\item Every Hausdorff zero-dimensional first-countable space of cardinality $< \mathfrak c$ can be densely embedded in a Hausdorff zero-dimensional first-countable countably compact space.
 \item Every regular first-countable space of weight $< \mathfrak c$ can be densely embedded in a regular first-countable countably compact space.
  \item Every regular first-countable space of cardinality $< \mathfrak c$ can be densely embedded in a regular first-countable countably compact space.
 \end{enumerate}
  \end{theorem}

Theorem~\ref{sbc} has the following zero-dimensional counterpart.

\begin{theorem} \label{sbc1} The following assertions are equivalent.   
\begin{enumerate}	
\item $\mathfrak b = \mathfrak s= \mathfrak c$.

\item Every Hausdorff zero-dimensional first-countable space of weight $< \mathfrak c$ can be densely embedded in a Hausdorff zero-dimensional first-countable countably compact space.
\item Every Hausdorff zero-dimensional first-countable space of cardinality $< \mathfrak c$ can be densely embedded in a Hausdorff zero-dimensional first-countable countably compact space.
 \end{enumerate}
  \end{theorem}  

Combining Theorem~\ref{sbc} with the results from~\cite{BZ} we get the following.

 \begin{theorem}\label{rigid} The following assertions are equivalent:
\begin{enumerate}
    \item $\mathfrak b = \mathfrak s= \mathfrak c$.
    \item Every regular separable first-countable non-normal space of weight $< \mathfrak c$ embeds into an $\mathbb R$-rigid Nyikos space.
     \item Every regular separable first-countable non-normal space of cardinality $< \mathfrak c$ embeds into an $\mathbb R$-rigid Nyikos  space.
\end{enumerate}   
 \end{theorem} 

Theorem~\ref{rigid} allows us to prove the following analogue of Theorem~\ref{BZZ} using a milder assumption. 
\begin{theorem}\label{BZZnew}
$(\w_1<\mathfrak b=\mathfrak s=\mathfrak c)$ There exists an $\mathbb R$-rigid  Nyikos space.   
\end{theorem}

%Recall that PFA implies $\w_1<\mathfrak b=\mathfrak s=\mathfrak c=\w_2$.
Theorems~\ref{NZ} and \ref{rigid}
imply the following corollary which shows a profound contrast in the behavior of normal and non-normal Nyikos spaces under PFA.  
\begin{corollary}\label{cormain}
 {\em (PFA)} The following assertions hold:
 \begin{enumerate}
     \item Each normal Nyikos space is compact.
     \item Each regular separable first-countable space of weight $<\mathfrak c$ embeds into an $\mathbb R$-rigid Nyikos space.
 \end{enumerate}
\end{corollary}

The following consistency result complements Theorem~\ref{BZ}.

\begin{theorem}\label{Tychonoff}
$(\heartsuit)$ Every Tychonoff first-countable space of weight $< \mathfrak c$ can be densely embedded into a Tychonoff first-countable countably compact space.    
\end{theorem}

However, the following question remains open:
\begin{question}
Can the assumption $(\heartsuit)$ be weakened to $\mathfrak b=\mathfrak s=\mathfrak c$ in Theorem~\ref{Tychonoff}?    
\end{question}

%  Note that assertion 7 of Theorem~\ref{sbc} is not useful in the case of CH, because each regular space of countable weight is normal. 

Turning to embeddings into pseudocompact spaces we obtain the following characterization.
  
\begin{theorem}\label{ps}
The following assertions are equivalent:
\begin{enumerate}
    \item $\mathfrak b = \mathfrak s= \mathfrak c$.
    \item  Every first-countable zero-dimensional Hausdorff space of weight $<\mathfrak c$ embeds densely into a first-countable zero-dimensional pseudocompact space. 
    \item  Every first-countable zero-dimensional Hausdorff space of cardinality $<\mathfrak c$ embeds densely into a first-countable zero-dimensional pseudocompact space. 
\end{enumerate}  
\end{theorem}

%Note that Theorems~\ref{st} and~\ref{tt} imply that zero-dimensionality is essential in Theorem~\ref{ps}.

Recall that a subspace $A$ of the Cantor space is called a {\em $\lambda$-set}  if each countable subset of $A$ is $G_{\delta}$.
As a by product we obtain the following characterization of $\lambda$-subsets of the Cantor space.

\begin{theorem}\label{PR}
A subspace $X$ of the Cantor space is a $\lambda$-set if and only if the Pixley-Roy hyperspace $\PR(X)$ embeds densely into a first-countable pseudocompact space.      
\end{theorem}

\section{Embedding theorems}

We start with embeddings into countably compact spaces.

As we mentioned in the previous section, the assumption $\mathfrak{s}=\mathfrak{b}=\mathfrak c$ is formally weaker than ($\heartsuit$). Indeed, fix a $P_{\mathfrak c}$-point $p$ and a splitting family $\mathcal C$. To derive a contradiction, assume that $|\mathcal C|<\mathfrak c$. Consider the family $\mathcal D=\mathcal C\cap p$ and find an arbitrary pseudointersection $E\in p$ of the family $\mathcal D\cup\{\omega\setminus C: C\in \mathcal C\setminus p\}$. It is easy to check that there exists no $C\in\mathcal C$ which splits $E$. This contradiction implies that $\mathfrak s=\mathfrak c$. On the other hand there exists a model of ZFC which satisfies $\mathfrak{s}=\mathfrak b=\mathfrak c$, but contains no $P_{\mathfrak c}$-points, see~\cite[Theorem 8]{BS}.
Hence the next theorem generalizes Theorem~\ref{BZ}.

\begin{proposition}\label{mainr}
$(\mathfrak{s} =\mathfrak {b}=\mathfrak c)$ Let $X$ be a regular first-countable space of weight $\kappa<\mathfrak{c}$. Then $X$ can be densely embedded into a regular first-countable countably compact space.
\end{proposition}

\begin{proof}
%Let $X$ be a first-countable regular space such that $w(X)<\mathfrak{c}=\mathfrak{s}=\mathfrak {b}$. 
Without loss of generality we can assume that the underlying set of $X$ is disjoint with $\mathfrak c$.  The first countability of $X$ implies that $|X|\leq \mathfrak c^{\w}=\mathfrak c$. If $X$ is countably compact, then there is nothing to prove. Otherwise, let 
$$\mathcal{D}=\{A\in [X]^{\omega}: \hbox{ } A \hbox{ is closed and discrete in }X \}.$$ Fix any bijection $h: \mathcal{D}\cup [\mathfrak{c}]^{\omega}\rightarrow \mathfrak{c}$ such that $h(a)\geq \sup(a)$ for any $a\in [\mathfrak{c}]^{\omega}$. It is easy to see that such a bijection exists. Next, for every $\alpha\leq \mathfrak{c}$ we shall recursively construct a topology $\tau_{\alpha}$ on $X_{\alpha}\subseteq X\cup\alpha$. For the sake of brevity we denote the space $(X_{\alpha},\tau_{\alpha})$ by $Y_{\alpha}$. At the end, we will show that the space $Y_{\mathfrak{c}}$ is regular first-countable countably compact and contains  $X$ as a dense subspace.

Let $X_0=X$. Since $X$ is first-countable and regular there exists a base $\mathcal B_0=\bigcup_{x\in X}B_{0}^x$ of the topology on $X$, where for each $x\in X$, the collection $\mathcal B_{0}^x=\{U_{n,0}^x:n\in \w\}$ is an open neighborhood base at $x$. With no loss of generality we can assume that $|\mathcal B_{0}|<\mathfrak c$; $U_{0,0}^x=X$ for each $x\in X$; and $\overline{U_{n+1,0}^x}\subset U_{n,0}^x$ for every $n\in\w$ and $x\in X$.

Assume that for each $\alpha<\xi$ regular first-countable spaces $Y_{\alpha}$ are already constructed by defining a base $\mathcal B_{\alpha}=\bigcup_{x\in X_{\alpha}}\mathcal B_{\alpha}^x$ of the the topology $\tau_{\alpha}$, where for each $x\in X_{\alpha}$, the collection $\mathcal B_{\alpha}^x=\{U_{n,\alpha}^x:n\in \w\}$ is an open neighborhood base at $x$. Additionally assume that $X_{\alpha}\subseteq X_{\beta}$ for any $\alpha\in\beta$ and the family $\mathcal B_{\alpha}$ satisfies the following conditions:
\begin{enumerate}
    \item $|\mathcal B_{\alpha}|<\mathfrak c$;
    \item $U_{0,\alpha}^x=X_{\alpha}$ for each $x\in X_{\alpha}$;
    \item $\cl_{Y_{\alpha}}(U_{n+1,\alpha}^x)\subset U_{n,\alpha}^x$ for every $n\in\w$ and $x\in X_{\alpha}$;
    %\item for every $x\in X_{\alpha}$ and $n\in\w$ there exists a continuous function $f_\frac{x}{n,\alpha}:Y_{\alpha}ction\rightarrow \mathbb R$ such that $f_{\frac{x}{n,\alpha}}^{-1}(0)=X_\alpha\setminus U_{n,\alpha}^x$.
    \item for every $\alpha<\beta<\xi$, $n\in\w$ and $x\in X_{\alpha}$ we have that  $U_{n,\alpha}^x=U_{n,\beta}^x\cap X_\alpha$.
\end{enumerate}

%for each $x\in X$ fix an open neighborhood base $\mathcal{B}_0^x=\{U_{n,0}^x\mid n\in\omega\}$ at the point $x$ such that $U_{0,0}^x=X$ and $\overline{U_{n+1,0}^x}\subset U_{n,0}^x$. 
%Also, for each $x\in X$ and $n\in\mathbb N$ fix any continuous real-valued function $f_{n,0}^x$ such that $f_{n,0}^x(x)=1$ and $f_{n,0}^x[(X\cup\alpha)\setminus U_{n,0}^x]=0$. Since the space $X$ is Tychonoff such functions exist.

%For any $x\in Y_{\alpha}$, $n\in\w$ and $\alpha<\beta<\xi$ we additionally assume the following 
%\begin{itemize}
    %\item $Y_{\alpha}$;
    %\item $\cl_{Y_{\alpha}}(U_{n+1,\alpha}^x)\subset U_{n,\alpha}^x$;
%    \item 
%    \item $|\mathbf B_{\eta}|<\mathfrak c$.
%\end{itemize}

There are three cases to consider:
\begin{itemize}
\item[1)] $\xi=\gamma+1$ for some $\gamma\in\mathfrak{c}$ and $h^{-1}(\gamma)\cap X_{\gamma}$ is not an infinite closed discrete subset of $Y_{\gamma}$;
\item[2)] $\xi=\gamma+1$ for some $\gamma\in\mathfrak{c}$ and $h^{-1}(\gamma)\cap X_{\gamma}$ is an infinite closed discrete subset of $Y_{\gamma}$;
\item[3)] $\xi$ is a limit ordinal.
\end{itemize}

\medskip

1) Let $X_{\xi}=X_{\gamma}$ and $\mathcal B_{\xi}=\mathcal B_{\gamma}$.

%For each $x\in X_{\xi}$ put $U_{0,\xi}^x=X_{\xi}$. For each $x\in X_{\gamma}$ and $n\in\N$ set $U_{n,\xi}^x=U_{n,\gamma}^x$. For each $n\in \mathbb N$ let $U_{n,\xi}^{\gamma}=\{\gamma\}$. It is straightforward to check that $\mathcal B_{\xi}=\{U_{n,\xi}^x: x\in X_{\xi}, n\in\w\}$ forms a base of some topology which we denote by $\tau_{\xi}$. Set $f_{U_{0,\xi}^\gamma}\equiv 1$ and for each $n\in\N$ let $f_{U_{n,\xi}^\gamma}$ to be a characteristic function of $\gamma$ on $X_{\xi}$. For each $x\in X_{\gamma}$ and $n\in \N$ put $f_{U_{n,\xi}^x}= f_{U_{n,\gamma}^x}\cup\{(\gamma, 0)\}$. Clearly, the families $\mathcal B_{\xi}$ and $\Phi_{\xi}$ satisfy conditions (1)--(6).
\medskip 

2) Put $X_{\xi}=X_{\gamma}\cup\{\gamma\}$. Let $h^{-1}(\gamma)=\{z_n\}_{n\in\omega}$. For each $U_{n,\gamma}^x\in\B_{\gamma}$ consider the set $$A_{U_{n,\gamma}^x}=\{k\in\w: z_k\in U_{n,\gamma}^x\setminus U_{n+1,\gamma}^x\}.$$ Since $|\B_{\gamma}|<\mathfrak c=\mathfrak s$ we get that the family $\{A_{U_{n,\gamma}^x}:U_{n,\gamma}^x\in\B_{\gamma}\}$ is not splitting, i.e. there exists a subset $A\subset \w$ such that for each $U_{n,\gamma}^x\in\B_{\gamma}$ the set 
$d_{\gamma}=\{z_n:n\in A\}$ is either almost contained in $U_{n,\gamma}^x\setminus U_{n+1,\gamma}^x$ or almost disjoint with $U_{n,\gamma}^x\setminus U_{n+1,\gamma}^x$. Since the set $\{z_n:n\in\w\}$ is closed in $Y_{\gamma}$ for each $x\in X_{\gamma}$ there exists $n\in\w$ such that $|U_{n,\gamma}^x\cap \{z_k:k\in\w\}|\leq 1$. Since $U^x_{0,\gamma}=X_{\gamma}$ for all $x\in X_{\gamma}$, we obtain that for each $x\in X_{\gamma}$ there exists a unique $m(x)\in\w$ such that $d_{\gamma}\subset^* U_{m(x),\gamma}^x\setminus U_{m(x)+1,\gamma}^x$.

For each $x\in X_\gamma $ let $\mathcal E(x)$ be the pair 
$(U^x_{m(x),\gamma}, U^x_{m(x)+2,\gamma})\in\B_\gamma{\times} \B_\gamma$.
Define a function $f_{\mathcal E(x)}\in\w^A$ as follows: if $z_n\in U^x_{m(x),\gamma}\setminus \cl_{Y_{\gamma}}(U^x_{m(x)+2,\gamma})$, then
$$f_{\mathcal E(x)}(n)=\min\{k: U^{z_n}_{k,\gamma}\subset U^x_{m(x),\gamma}\setminus \cl_{Y_{\gamma}}(U^x_{m(x)+2,\gamma})\},$$  and $f_{\mathcal E(x)}(n)=0$, otherwise. %Note that the function $f_{U^x_{n,\gamma}}$ is well-defined, as $\overline{U^x_{n+1,\gamma}}\subset U^x_{n,\gamma}$.
Since $|\B_{\gamma}|<\mathfrak c=\mathfrak b$ there exists a function $f\in\w^A$ such that $f\geq^* f_{\mathcal E(x)}$ for every $U^x_{m(x),\gamma}\in \B_{\gamma}$.
With no loss of generality we can additionally assume that $U_{f(n),\gamma}^{z_n}\cap U_{f(m),\gamma}^{z_m}=\emptyset$ for each distinct $n,m\in A$. It is easy to see that the sequence $\{U_{f(n),\gamma}^{z_n}:n\in A\}$ is locally finite in $Y_{\gamma}$.

Next we define an open neighborhood base at the point $\gamma\in X_{\xi}$: Put 
$$U_{0,\xi}^{\gamma}=X_{\xi}\quad  \hbox{ and }\quad U_{k,\xi}^{\gamma}=\bigcup_{n\in A\setminus k}U^{z_n}_{f(n)+k,\gamma}\cup\{\gamma\}$$ for all $k\in \N$. 

Now we are going to define an open neighborhood base $\mathcal B_{\xi}^x$ at each point $x\in X_{\gamma}$. Let $U_{0,\xi}^x=X_{\xi}$. 
% for every $x\in X_{\gamma}$. 
%For each $x\in X_{\gamma}$ and 
For each $n\in\N$ let $U_{n,\xi}^x=U_{n, \gamma}^x$ if $n\geq m(x)+1$ and $U_{n,\xi}^x=U_{n,\gamma}^x\cup\{\gamma\}$ if $n\leq m(x)$. 
It is easy to check that the family $\mathcal B_{\xi}=\{U_{n,\xi}^x: x\in X_{\xi},n\in\w\}$ forms a base of a topology $\tau_{\xi}$, and for each $x\in X_{\xi}$ the family $\mathcal B_{\xi}^{x}=\{U_{n,\xi}^x:n\in\w\}$ forms an open neighborhood base at $x$ in $Y_{\xi}$. 

%For each $U_{k,\gamma}^x\in \mathcal B_{\gamma}$ let $b^k_x=\lim_{n\in F_\gamma} f_{U_{k,\gamma}^x}(z_n)$, which exists by the choice of $F_{\gamma}$.

At this point it is a tedious routine to check that the family $\mathcal B_{\xi}$  satisfies conditions (1)--(4).

\medskip

3) Let $X_{\xi}=\bigcup_{\alpha\in \xi}X_{\alpha}$. For each $x\in X_{\xi}$ let $\theta_x=\min\{\alpha:x\in X_{\alpha}\}$, and for every $n\in\w$ let us put $$U_{n,\xi}^x=\bigcup_{\theta_x\leq\alpha<\xi}U_{n,\alpha}^x.$$
It is quite straightforward to check that the family $\mathcal B_{\xi}=\{U_{n,\xi}^x: n\in\w, x\in X_{\xi}\}$ forms a base of a topology $\tau_{\xi}$, and the family $\mathcal B_{\xi}$ satisfies conditions (1), (2) and (4). In order to check the validity of condition (3) we need the following auxiliary claim.

\begin{claim}\label{clr}
For any $\gamma\in\xi$, $n,m\in\w$ and distinct points $y_0, y_1\in Y_{\gamma}$, if $U^{y_0}_{n,\gamma}\cap U^{y_1}_{m,\gamma}=\emptyset$, then $U^{y_0}_{n,\xi}\cap U^{y_1}_{m,\xi}=\emptyset$.
\end{claim}

\begin{proof} Seeking a contradiction, assume that $U^{y_0}_{n,\xi}\cap U^{y_1}_{m,\xi}\neq\emptyset$, but $U^{y_0}_{n,\gamma}\cap U^{y_1}_{m,\gamma}=\emptyset$ for some $\gamma\in \xi$. It is easy to see that $U^{y_0}_{n,\xi}\cap U^{y_1}_{m,\xi}\subset\xi\setminus \gamma$. Let $\delta=\min U^{y_0}_{n,\xi}\cap U^{y_1}_{m,\xi}$. It follows that $U^{y_0}_{n,\delta}\cap U^{y_1}_{m,\delta}=\emptyset$ and $U^{y_0}_{n,\delta+1}\cap U^{y_1}_{m,\delta+1}=\{\delta\}$. Then the set $d_{\delta}$ (see case 2 above) is closed and discrete in $Y_{\delta}$. Since $\delta\in U^{y_0}_{n,\delta+1}\cap U^{y_1}_{m,\delta+1}$, the definition of $U^{y_0}_{n,\delta+1}$ and $U^{y_1}_{m,\delta+1}$ implies that $d_{\delta}\subset^*U^{y_0}_{n,\delta}\cap U^{y_1}_{m,\delta}=\emptyset$, which contradicts our assumption.
\end{proof}

Fix any $x\in X_{\xi}$, $n\in \w$ and $z\in \overline{U_{n+1,\xi}^x}$. There exists an ordinal $\gamma\in \xi$ such that $x,z\in X_{\gamma}$. We claim that $z\in\cl_{Y_{\gamma}}(U_{n+1,\gamma}^x)$. Indeed, otherwise, there exists $m\in\w$ such that $U^{z}_{m,\gamma}\cap U_{n+1,\gamma}^x=\emptyset$. The above claim implies that $U^{z}_{m,\xi}\cap U_{n+1,\xi}^x=\emptyset$, which contradicts the choice of $z$. Thus $z\in \cl_{Y_{\gamma}}(U_{n+1,\gamma}^x)\subset U_{n,\gamma}^x\subset U^x_{n,\xi}$, which establishes condition (3).

By the construction, the space $Y_{\mathfrak{c}}$ is regular, first-countable and contains $X$ as a dense subspace. Let $A$ be any countable subset of $Y_{\mathfrak{c}}$. If the set $B=A\cap \mathfrak{c}$ is infinite, then consider $h(B)\in\mathfrak{c}$. By the construction, either $B$ has an accumulation point in $Y_{h(B)}$ or $h(B)$ is an accumulation point of $B$ in $Y_{h(B)+1}$. In both cases $B$ has an accumulation point in $Y_{\mathfrak c}$. If $A\subset^* X$, then either it has an accumulation point in $X$, or $A$ is closed and discrete in $X$. In the latter case either $A$ has an accumulation point in $Y_{h(A)}$ or $h(A)$ is an accumulation point of $A$ in $Y_{h(A)+1}$. Thus $Y_{\mathfrak c}$ is countably compact. 
\end{proof}

\begin{proposition}\label{mainz}
$(\mathfrak{s} =\mathfrak{b}=\mathfrak c)$ Let $X$ be a Hausdorff zero-dimensional first-countable space of weight $\kappa<\mathfrak{c}$. Then $X$ embeds densely into a Hausdorff zero-dimensional first-countable countably compact space.
\end{proposition}

\begin{proof}
The proof of this theorem is very similar to one of Proposition~\ref{mainr}, so we give only a sketch of it.
%Let $X$ be a first-countable Hausdorff zero-dimensional non countably compact space such that $w(X)<\mathfrak{c}=\mathfrak{s}=\mathfrak{b}$. 
Let $\mathcal{D}$ and $h$ be such as in the proof of Proposition~\ref{mainr}. 
%$$\mathcal{D}=\{A\in [X]^{\omega}: \hbox{ } A \hbox{ is closed and discrete in }X \}.$$ Fix any bijection $h: \mathcal{D}\cup [\mathfrak{c}]^{\omega}\rightarrow \mathfrak{c}$ such that $h(a)\geq \sup(a)$ for any $a\in [\mathfrak{c}]^{\omega}$. It is easy to see that such a bijection exists. 
%For every $\alpha\leq \mathfrak{c}$ we shall recursively construct a topology $\tau_{\alpha}$ on $X_{\alpha}\subset X\cup\alpha$. For the sake of brevity we denote the space $(X_{\alpha},\tau_{\alpha})$ by $Y_{\alpha}$. %At the end, we will show that the space $Y_{\mathfrak{c}}$ has the desired properties.

Let $X_0=X$. Since $X$ is first-countable and zero-dimensional there exists a base $\mathcal B_0=\bigcup_{x\in X}\mathcal B_{0}^x$ of the topology on $X$ consisting of clopen sets, where for each $x\in X$, the collection $\mathcal B_{0}^x=\{U_{n,0}^x:n\in \w\}$ is a nested open neighborhood base at $x$. With no loss of generality we can assume that $|\mathcal B_{0}|<\mathfrak c$ and $U_{0,0}^x=X$ for each $x\in X$.

Assume that for each $\alpha<\xi$ Hausdorff zero-dimensional first-countable spaces $Y_{\alpha}$ are already constructed by defining a base $\mathcal B_{\alpha}=\bigcup_{x\in X_{\alpha}}\mathcal B_{\alpha}^x$ of the topology $\tau_{\alpha}$ on a set $X_{\alpha}\subseteq X\cup\alpha$, where for each $x\in X_{\alpha}$, the collection $\mathcal B_{\alpha}^x=\{U_{n,\alpha}^x:n\in \w\}$ is a nested open neighborhood base at $x$ consisting of clopen sets. Additionally assume that $X_{\alpha}\subseteq X_{\beta}$ for any $\alpha\in\beta$ and the family $\mathcal B_{\alpha}$ satisfies the following conditions:
\begin{enumerate}
    \item $|\mathcal B_{\alpha}|<\mathfrak c$;
    \item $U_{0,\alpha}^x=X_{\alpha}$ for each $x\in X_{\alpha}$;
   % \item $\cl_{Y_{\alpha}}(U_{n+1,\alpha}^x)\subset U_{n,\alpha}^x$ for every $n\in\w$ and $x\in X_{\alpha}$;
    \item for every $\alpha<\beta<\xi$, $n\in\w$ and $x\in X_{\alpha}$ we have that  $U_{n,\alpha}^x=U_{n,\beta}^x\cap X_\alpha$.
\end{enumerate}

%for each $x\in X$ fix an open neighborhood base $\mathcal{B}_0^x=\{U_{n,0}^x\mid n\in\omega\}$ at the point $x$ such that $U_{0,0}^x=X$ and $\overline{U_{n+1,0}^x}\subset U_{n,0}^x$. 
%Also, for each $x\in X$ and $n\in\mathbb N$ fix any continuous real-valued function $f_{n,0}^x$ such that $f_{n,0}^x(x)=1$ and $f_{n,0}^x[(X\cup\alpha)\setminus U_{n,0}^x]=0$. Since the space $X$ is Tychonoff such functions exist.

%For any $x\in Y_{\alpha}$, $n\in\w$ and $\alpha<\beta<\xi$ we additionally assume the following 
%\begin{itemize}
    %\item $Y_{\alpha}$;
    %\item $\cl_{Y_{\alpha}}(U_{n+1,\alpha}^x)\subset U_{n,\alpha}^x$;
%    \item 
%    \item $|\mathbf B_{\eta}|<\mathfrak c$.
%\end{itemize}

There are three cases to consider:
\begin{itemize}
\item[1)] $\xi=\gamma+1$ for some $\gamma\in\mathfrak{c}$ and $h^{-1}(\gamma)\cap X_{\gamma}$ is not an infinite closed discrete subset of $Y_{\gamma}$;
\item[2)] $\xi=\gamma+1$ for some $\gamma\in\mathfrak{c}$ and $h^{-1}(\gamma)\cap X_{\gamma}$ is an infinite closed discrete subset of $Y_{\gamma}$;
\item[3)] $\xi$ is a limit ordinal.
\end{itemize}

\medskip

1) Let $X_{\xi}=X_{\gamma}$ and $\mathcal B_{\xi}=\mathcal B_{\gamma}$.

%For each $x\in X_{\xi}$ put $U_{0,\xi}^x=X_{\xi}$. For each $x\in X_{\gamma}$ and $n\in\N$ set $U_{n,\xi}^x=U_{n,\gamma}^x$. For each $n\in \mathbb N$ let $U_{n,\xi}^{\gamma}=\{\gamma\}$. It is straightforward to check that $\mathcal B_{\xi}=\{U_{n,\xi}^x: x\in X_{\xi}, n\in\w\}$ forms a base of some topology which we denote by $\tau_{\xi}$. Set $f_{U_{0,\xi}^\gamma}\equiv 1$ and for each $n\in\N$ let $f_{U_{n,\xi}^\gamma}$ to be a characteristic function of $\gamma$ on $X_{\xi}$. For each $x\in X_{\gamma}$ and $n\in \N$ put $f_{U_{n,\xi}^x}= f_{U_{n,\gamma}^x}\cup\{(\gamma, 0)\}$. Clearly, the families $\mathcal B_{\xi}$ and $\Phi_{\xi}$ satisfy conditions (1)--(6).
\medskip 

2) Put $X_{\xi}=X_{\gamma}\cup\{\gamma\}$. Let $h^{-1}(\gamma)=\{z_n\}_{n\in\omega}$. Similarly as in the proof of Theorem~\ref{mainr}, one can find a subset $A\subset \w$ such that for each $U_{n,\gamma}^x\in\B_{\gamma}$ the set 
$d_{\gamma}=\{z_n:n\in A\}$ is either almost contained in $U_{n,\gamma}^x\setminus U_{n+1,\gamma}^x$ or almost disjoint with $U_{n,\gamma}^x\setminus U_{n+1,\gamma}^x$. %Using similar arguments as in the proof of Theorem~\ref{mainr} it can be shown that
Thus, for each $x\in X_{\gamma}$ there exists a unique $m(x)\in\w$ such that $d_{\gamma}\subset^* U_{m(x),\gamma}^x\setminus U_{m(x)+1,\gamma}^x$.

For each $x\in X_\gamma$ let us denote by $\mathcal E(x)$ the pair
$(U^x_{m(x),\gamma}, U^x_{m(x)+1,\gamma})\in\mathcal B_\gamma{\times}\mathcal B_\gamma$.
Define a function $f_{\mathcal E(x)}\in\w^A$ as follows: if $z_n\in U^x_{m(x),\gamma}\setminus U^x_{m(x)+1,\gamma}$, then $$f_{\mathcal E(x)}(n)=\min\{k: U^{z_n}_{k,\gamma}\subset U^x_{m(x),\gamma}\setminus U^x_{m(x)+1,\gamma}\},$$  and $f_{\mathcal E(x)}(n)=0$, otherwise. Note that $f_{\mathcal E(x)}$ is well-defined, as the set $U^x_{m(x),\gamma}\setminus U_{m(x)+1,\gamma}^x$ is open for every $x\in X_{\gamma}$.Since $|\B_{\gamma}|<\mathfrak c=\mathfrak b$ we get that there exists a function $f\in\w^A$ such that $f\geq^* f_{\mathcal E(x)}$ for every $\mathcal E(x)\in \B_{\gamma}{\times}\mathcal B_\gamma$.
We can additionally assume that $U_{f(n),\gamma}^{z_n}\cap U_{f(m),\gamma}^{z_m}=\emptyset$ for each distinct $n,m\in A$. It is easy to see that the family $\{U_{f(n),\gamma}^{z_n}:n\in A\}$ is locally finite in $Y_{\gamma}$.

Next we define an open neighborhood base at the point $\gamma\in X_{\xi}$: Put 
$$U_{0,\xi}^{\gamma}=X_{\xi}\quad  \hbox{ and }\quad U_{k,\xi}^{\gamma}=\bigcup_{n\in A\setminus k}U^{z_n}_{f(n)+k,\gamma}\cup\{\gamma\}$$ for all $k\in \N$. Since the family $\{U_{f(n),\gamma}^{z_n}:n\in A\}$ is locally finite, the set $U_{k,\xi}^{\gamma}$ is clopen for each $k\in\w$.

Now we are going to define an open neighborhood base $\mathcal B_{\xi}^x$ at each point $x\in X_{\gamma}$. Let $U_{0,\xi}^x=X_{\xi}$ for every $x\in X_{\gamma}$. 
For each $x\in X_{\gamma}$ and $n\in\N$ let $U_{n,\xi}^x=U_{n, \gamma}^x$ if $n\geq m(x)+1$ and $U_{n,\xi}^x=U_{n,\gamma}^x\cup\{\gamma\}$ if $n\leq m(x)$. 
It is easy to check that the family $\mathcal B_{\xi}=\{U_{n,\xi}^x: x\in X_{\xi},n\in\w\}$ forms a base of a topology $\tau_{\xi}$, and for each $x\in X_{\xi}$ the family $\mathcal B_{\xi}^{x}=\{U_{n,\xi}^x:n\in\w\}$ forms a nested open neighborhood base at $x$ in $Y_{\xi}$. Moreover, the space $Y_{\xi}$ is Hausdorff and zero-dimensional.

%For each $U_{k,\gamma}^x\in \mathcal B_{\gamma}$ let $b^k_x=\lim_{n\in F_\gamma} f_{U_{k,\gamma}^x}(z_n)$, which exists by the choice of $F_{\gamma}$.

At this point it is a tedious routine to check that the family $\mathcal B_{\xi}$ satisfies conditions (1)--(3).

\medskip

3) Let $X_{\xi}=\bigcup_{\alpha\in \xi}X_{\alpha}$. For each $x\in X_{\xi}$ let $\theta_x=\min\{\alpha:x\in X_{\alpha}\}$. For each $x\in X_{\xi}$ and $n\in\w$ put $$U_{n,\xi}^x=\bigcup_{\theta_x\leq\alpha<\xi}U_{n,\alpha}^x.$$
It is straightforward to check that the family $\mathcal B_{\xi}=\{U_{n,\xi}^x: n\in\w, x\in X_{\xi}\}$ forms a base of a topology $\tau_{\xi}$, satisfies conditions (1)--(3) and  for each $x\in X_{\xi}$ the family $\mathcal B_{\xi}^{x}=\{U_{n,\xi}^x:n\in\w\}$ forms an open neighborhood base at $x$ in $Y_{\xi}$.
%It remains to check that $Y_{\xi}$ is a Hausdorff zero-dimensional space.
Hausdorffness of the space $Y_{\xi}$ follows from the next claim which can be proved similarly as Claim~\ref{clr}.
\begin{claim} For any $\gamma\in\xi$, $n,m\in\w$ and distinct points $y_0, y_1\in Y_{\gamma}$, if $U^{y_0}_{n,\gamma}\cap U^{y_1}_{m,\gamma}=\emptyset$, then $U^{y_0}_{n,\xi}\cap U^{y_1}_{m,\xi}=\emptyset$.
\end{claim}
%\begin{proof} Assume to the contrary that $U^{y_0}_{n,\xi}\cap %U^{y_1}_{m,\xi}\neq\emptyset$, but $U^{y_0}_{n,\gamma}\cap U^{y_1}_{m,\gamma}=\emptyset$ for some $\gamma\in \xi$. It is easy to see that $U^{y_0}_{n,\xi}\cap U^{y_1}_{m,\xi}\subset\xi\setminus \gamma$. Let $\delta=\min U^{y_0}_{n,\xi}\cap U^{y_1}_{m,\xi}$. It follows that $U^{y_0}_{n,\delta}\cap U^{y_1}_{m,\delta}=\emptyset$ and $U^{y_0}_{n,\delta+1}\cap U^{y_1}_{m,\delta+1}=\{\delta\}$. Then the set $d_{\delta}$ (see case 2 above) is closed and discrete in $Y_{\delta}$. Since $\delta\in U^{y_0}_{n,\delta+1}\cap U^{y_1}_{m,\delta+1}$, the definition of $U^{y_0}_{n,\delta+1}$ and $U^{y_1}_{m,\delta+1}$ implies that $d_{\delta}\subset^*U^{y_0}_{n,\delta}\cap U^{y_1}_{m,\delta}=\emptyset$, which contradicts our assumption.
%\end{proof}

It remains to check that the space $Y_{\xi}$ is zero-dimensional which is done in the following claim.

\begin{claim} For any $x\in X_\xi$ and $n\in\w$ the set $U^{x}_{n,\xi}$ is closed.
\end{claim}
\begin{proof} 
Fix any $y\in X_{\xi}\setminus U^{x}_{n,\xi}$ and find $\gamma<\xi$ such that $x,y\in X_{\gamma}$. It follows that $y\notin U^{x}_{n,\gamma}$. By the inductive assumption the set $U^{x}_{n,\gamma}$ is closed, implying the existence of $m\in \w$ such that $U^{y}_{m,\gamma}\cap U^{x}_{n,\gamma}=\emptyset$. Then the previous claim implies that $U^{y}_{m,\xi}\cap U^{x}_{n,\xi}=\emptyset$, witnessing that the set $U^{x}_{n,\xi}$ is closed.
\end{proof}

Similarly as in the proof of Theorem~\ref{mainr} it can be checked that the space $Y_{\mathfrak{c}}$ is Hausdorff, zero-dimensional, first-countable, countably compact and contains $X$ as a dense subspace.
%By the construction, the space $Y_{\mathfrak{c}}$ is Hausdorff, zero-dimensional, first-countable and contains $X$ as a dense subspace. Let $A$ be any countable subset of $Y_{\mathfrak{c}}$. If the set $B=A\cap \mathfrak{c}$ is infinite, then consider $h(B)\in\mathfrak{c}$. By the construction, either $B$ has an accumulation point in $Y_{h(B)}$ or $h(B)$ is an accumulation point of $B$ in $Y_{h(B)+1}$. In both cases $B$ has an accumulation point in $Y_{\mathfrak c}$. If $A\subset^* X$, then either it has an accumulation point in $X$, or $A$ is closed and discrete in $X$. In the latter case either $A$ has an accumulation point in $Y_{h(A)}$ or $h(A)$ is an accumulation point of $A$ in $Y_{h(A)+1}$. Thus $Y_{\mathfrak c}$ is countably compact. 
\end{proof}

Let $D$ be a subset of a set $S$. An {\em expansion} of $D$ is a family $\{A_d : d \in D\}$ of subsets of $S$ such that $d \in A_d$ and
$A_d \cap D = \{d\}$ for all $d \in D$. A family $U$ of subsets of a space $X$ is called {\em discrete} if every point of $X$ has an open neighborhood intersecting at most one element of $U$.  A topological space $X$  satisfies (or has) {\em Property D} if every
countable closed discrete subspace $D \subset X$ has an expansion to a discrete family of open sets. A space $X$ is called {\em pseudonormal}  if, whenever $F_0$ and $F_1$ are disjoint closed sets, one of which 
is countable, there are disjoint open sets $U_0$ and $U_1$ containing $F_0$ and $F_1$ respectively. 
%It is easy to check that every pseudonormal space has property D. 
We will also need the following facts proved in~\cite[Proposition 3.6]{Ny1} and~\cite[Theorem 3.7]{Ny1}, respectively.

\begin{proposition}[{Nyikos}]\label{propD}
Every pseudonormal space $X$ has property D.    
\end{proposition}

\begin{proposition}[{Nyikos}]\label{Lindelof}
Every first-countable regular space of Lindel\"of number $<\mathfrak b$ is pseudonormal.
 \end{proposition}

Propositions~\ref{propD} and~\ref{Lindelof} imply the following useful corollary.

 \begin{corollary}\label{useful}
 Every first-countable regular space of weight $<\mathfrak b$ has property D.    
 \end{corollary}

The proof of the following result uses techniques developed in~\cite{A,Ny1,Ny2,Ost}. 

 \begin{proposition}\label{lce}
$(\mathfrak b=\mathfrak c)$ Each first-countable Hausdorff locally compact space of weight $<\mathfrak c$ embeds densely into a first-countable Hausdorff locally compact countably compact space.  
 \end{proposition}
\begin{proof}
  Fix any Hausdorff locally compact first-countable space $X$ of weight $< \mathfrak c=\mathfrak b$. Without loss of generality we can assume that the underlying set of $X$ is disjoint with $[0,1){\times}\mathfrak c$. First countability of $X$ implies that $|X|\leq \mathfrak c^{\w}=\mathfrak c$. If $X$ is countably compact, then there is nothing to prove. Otherwise, let 
$$\mathcal{D}_0=\{A\in [X]^{\omega}: \hbox{ } A \hbox{ is closed and discrete in }X \}.$$ 

Fix any bijection $h:\mathfrak c{\times} \mathfrak c\rightarrow \mathfrak c$ such that $h((a,b))\geq a$. %whenever $a>0$.
% such that for every $A=\{(a_i,b_i):i\in\w\}\in[\mathfrak c{\times} \mathfrak c]^{\w}$ we have that $h(A)>\sup\{a_i:i\in \w\}$ if $\sup\{a_i:i\in \w\}>0$. It is easy to see that such a bijection exists.

 Next, for every $\alpha\leq \mathfrak{c}$ we shall recursively construct a topology $\tau_{\alpha}$ on $X_{\alpha}$ which will be a union of $X$ and some pairwise disjoint family of copies of the half interval $[0,1)$. For the sake of brevity we denote the space $(X_{\alpha},\tau_{\alpha})$ by $Y_{\alpha}$. At the end, we will show that the space $Y_{\mathfrak{c}}$ has the desired properties.  
 
% For each $\alpha\in\mathfrak c$ let $[0,1)_\alpha=\{x_\alpha:x\in [0,1)\}$  be a copy of the half interval $[0,1)$. We assume that $[0,1)_{\alpha}\cap [0,1)_{\beta}=\emptyset$ for each $\alpha\neq \beta$. Also, it is convenient to denote a subset $\{x_{\alpha}: x\in A\subset [0,1)\}\subset [0,1)_{\alpha}$ by $A_{\alpha}$.
 
 Let $X_0=X$ and fix any bijection $g_0: \mathcal{D}_0\rightarrow \{0\}{\times}\mathfrak{c}$. Note that by the definition of $h$,
 $h^{-1}(0)\in\{0\}{\times}\mathfrak c$, so
 $g_0^{-1}(h^{-1}(0))$ is well-defined (and belongs to $\in \mathcal D_0$). By Corollary~\ref{useful}, the space $X$ has property $D$. It follows that the closed discrete set $\{z_n:n\in\w\}=g_0^{-1}(h^{-1}(0))$ can be expanded to a discrete family $\{U_n:n\in\w\}$ of open sets such that $z_n\in U_n$. Since $X$ is locally compact, we lose no generality assuming that $\cl_{X}(U_n)$ is compact for every $n\in\w$. For each $n\in\w$ fix a continuous function $f_n$ such that $f(z_n)=0$ and $f(X\setminus U_n)=1$.  Let $X_1=X_0\cup([0,1){\times}\{1\})$, and $\tau_1$ be a topology on $X_1$ which satisfies the following conditions:
 \begin{enumerate}
     \item $X_0$ is an open subspace of $X_1$;
     \item  the sets $$W_n(0,1)=[0,1/n){\times}\{1\}\cup \bigcup_{m\geq n}f_m^{-1}([0,1/n)),$$ $n\in\mathbb N$ form an open neighborhood base at the point $(0,1)\in X_1$;
     \item For $0<x<1$ the sets 
     $$W_n(x,1)=(x-1/n,x+1/n){\times}\{1\}\cup\bigcup_{m\geq n}f_m^{-1}((x-1/n,x+1/n)),$$ where $n\in\mathbb N$ satisfies $(x-1/n,x+1/n)\subseteq (0,1)$, form an open neighborhood base at the point $(x,1)\in X_1$. 
 \end{enumerate}
 It is easy to check that $Y_1$ is a first-countable Hausdorff space of weight $<\mathfrak b$.
 Observe that $(0,1)$ is an accumulation point of the set $\{z_n:n\in\w\}=g_0^{-1}(h^{-1}(0))$ in $Y_1$.
 Clearly, $Y_1$ is locally compact at each point $x\in X_0$. Let us show that $Y_1$ is locally compact at $(0,1)$. We claim that the set $\cl_{Y_1}(W_2(0,1))$ is compact. Since the family $\{U_n:n\in\w\}$ is discrete and the sets $U_n$, $n\in\w$ have compact closure in $X$, it is easy to see that the set 
 $$\cl_{Y_1}(W_2(0,1))\subseteq [0,1/2]{\times}\{1\}\cup\bigcup_{n\in\w\setminus\{0,1\}}\cl_X(U_n)$$ 
 is $\sigma$-compact. At this point it is enough to show that $\cl_{Y_1}(W_2(0,1))$ is countably compact. Fix any infinite subset $A=\{a_n:n\in\w\}\subset \cl_{Y_1}(W_2(0,1))$. If for some $n\in\w$ the set $A\cap\cl_X(U_n)$ is infinite or the set $A\cap ([0,1/2){\times}\{1\})$ is infinite, then $A$ has an accumulation point in $Y_1$. 
 %Suppose that the set $A\cap ([0,1/2){\times}\{1\})$ is finite and $A\cap\cl_X(U_n)$ is finite for all $n\in\w$. 
 So, let us assume that all the aforementioned intersections are finite.
 Shrinking the set $A$ if necessary, we can assume that $A\subset X$ and $|A\cap U_n|\leq 1$ for all $n\in\w$. For each $n\in \w$ let $m(n)$ be such that $a_n\in U_{m(n)}$. Put $y_n=f_{m(n)}(a_n)$. Since $\{y_n:n\in \w\}\subset [0,1/2)$ and the interval $[0,1/2]$ is compact there exists $z\in [0,1/2]$ and a subsequence $\{y_{n_k}:k\in\w\}$ which converges to $z$. It is easy to see that $(z,1)$ is an accumulation point of the set $A$ in $Y_1$. Hence the $\sigma$-compact set $\cl_{Y_1}(W_2(0,1))$ is countably compact and thus compact. Similarly one can check that $Y_1$ is locally compact at $(x,1)$ for each $0<x<1$. Let
 $$\mathcal{D}_1=\{A\in [X_1]^{\omega}: \hbox{ } A \hbox{ is closed and discrete in }Y_1 \}.$$ 
Fix any bijection $g_1: \mathcal{D}_1\rightarrow \{1\}{\times}\mathfrak{c}$. 

 Assume that locally compact Hausdorff spaces $Y_\alpha=(X_{\alpha},\tau_{\alpha})$ with weight $<\mathfrak b$ are constructed for all $\alpha\in\xi$ such that $Y_\alpha$ is an open dense subspace of $Y_{\eta}$ for all $\alpha\in \eta\in \xi$. Additionally assume that if $X_{\alpha+1}\setminus X_{\alpha}\neq \emptyset$, then $X_{\alpha+1}\setminus X_{\alpha}=[0,1){\times}\{\alpha+1\}$ and for each $\alpha\in\xi$ it is defined a bijection $g_{\alpha}: \mathcal{D}_{\alpha}\rightarrow \{\alpha\}{\times}\mathfrak{c}$, where $$\mathcal{D}_{\alpha}=\{A\in [X_{\alpha}]^{\omega}: \hbox{ } A \hbox{ is closed and discrete in }Y_{\alpha} \}.$$

% There are two cases to consider:
%\begin{itemize}
%\item[1)] $\xi$ is a limit ordinal;
%\item[2)] $\xi=\gamma+1$ for some $\gamma\in\mathfrak{c}$.
%\item[2)] $\xi=\gamma+1$ for some $\gamma\in\mathfrak{c}$ and $h^{-1}(\gamma)$ is an infinite closed discrete subset of $Y_{\gamma}$;
%\end{itemize}

If $\xi$ is a limit ordinal, then put $X_{\xi}=\bigcup_{\alpha\in\xi}X_{\alpha}$ and $\tau_{\xi}=\bigcup_{\alpha\in\xi}\tau_{\alpha}$. Clearly, $Y_\xi$ is first-countable, Hausdorff, locally compact with weight less than $\mathfrak b$. 
%Let $$\mathcal{D}_\xi=\{A\in [X_{\xi}]^{\omega}: \hbox{ } A \hbox{ is closed and discrete in }Y_\xi \}.$$ 
Fix any bijection $g_{\xi}: \mathcal D_\xi\rightarrow \{\xi\}{\times}\mathfrak c$.

Assume that $\xi=\gamma+1$ for some $\gamma\in\mathfrak{c}$. Consider the pair $h^{-1}(\gamma)=(a,b)\in \mathfrak c{\times}\mathfrak c$. If the set $g_a^{-1}(h^{-1}(\gamma))\in \mathcal D_a$ already has an accumulation point in $Y_{\gamma}$, then  put $Y_{\xi}=Y_{\gamma}$. If the set $g_a^{-1}(h^{-1}(\gamma))$ is closed and discrete in $Y_{\gamma}$, then we construct a Hausdorff locally compact first-countable space $Y_{\xi}$ of weight $<\mathfrak b$ similarly as we constructed $Y_1$. In particular, $X_{\xi}=X_{\gamma}\cup[0,1){\times}\{\xi\}$, and $(0,\xi)$ is an accumulation point of the set $g_a^{-1}(h^{-1}(\gamma))$.

Obviously, $Y_{\mathfrak c}$ is a Hausdorff locally compact first-countable space. To derive a contradiction, assume that $Y_{\mathfrak c}$ is not countably compact. Then $Y_{\mathfrak c}$ contains an infinite countable discrete closed subset $A$. Since the cardinal $\mathfrak b=\mathfrak c$ is regular, there exists  an ordinal $\xi\in\mathfrak c$ such that $A\subset X_{\xi}$. Obviously, $A$ is an infinite closed discrete subset of $Y_{\beta}$ for each $\beta\geq \xi$. Then $g_{\xi}(A)=(\xi,\delta)$ for some $\delta\in\mathfrak c$. Let $\mu\in \mathfrak c$ be such that $\xi\leq \mu$ and $h((\xi,\delta))=\mu$. By the construction, the point $(0,\mu+1)\in X_{\mu+1}$ is an accumulation point of $A$ in $Y_{\mathfrak c}$, which implies the desired contradiction. Hence the space $Y_{\mathfrak c}$ is countably compact.   
\end{proof}

Next we turn to embeddings into pseudocompact spaces.
We adopt the notations from~\cite{Bell}.

Let $S=\{S_n:n\in\w\}$ and $T=\{T_n:n\in\w\}$ be sequences of subsets of a set $X$. We say that $S$ {\em refines} $T$ if there exists an injection $f\in\w^\w$ such that $S_n\subseteq T_{f(n)}$ for all $n\in\w$. We write that the sequence $S$ is {\em eventually contained} in a subset $A\subset X$ if there exists $k\in\w$ such that $\bigcup_{n\geq k}S_n\subset A$.  The sequences $S$ and $T$ are called {\em eventually disjoint} if there exists $k\in\w$ such that $\bigcup_{n\geq k}S_n\cap \bigcup_{n\geq k}T_n=\emptyset$. A family $\mathcal R$ of sequences of subsets of $X$ is called {\em eventually disjoint} if any distinct elements $S,T\in\mathcal R$ are eventually disjoint. 

Let $X$ be a zero-dimensional space and $\B$ be a base of $X$ consisting of clopen sets. By $\D(\B)$ we denote the set of all locally finite sequences of elements of $\B$. A sequence $S=\{S_n:n\in\w\}\in\D(\B)$ is called {\em $\B$-exact} if for each $B\in\B$ which intersects infinitely many $S_n$ there exists $k\in\w$ such that $\bigcup_{n\geq k}S_n\subset B$. 
Let $\E(\B)=\{S\in\D(\B): S$ is $\B$-exact$\}$.
%and for each $S\in \E(\B)$ by $\mathbf{S}$ we denote a collection $\{\{S^n_m:m\in\w\}:n\in \w\}\subset \E(\B)$ such that $S=\{S^0_n:n\in\w\}$ and $\overline{S^{n+1}_m}\subset S^{n}_m$ for every $n,m\in\w$. Clearly, the family $\mathbf{S}$ is not uniquely defined. 

For any eventually disjoint family $\mathcal R\subset \E(\B)$ %and for each choice of the collection $\{\mathbf{R}=\{R_n^m: n,m\in\w\}:R\in\mathcal R\}$ 
we define a topology $\tau$ on the set $X\cup \mathcal R$ as follows: for each $B\in\B$ let $B^*=B\cup\{R\in\mathcal R: R$ is eventually contained in $B\}$. For every $R=\{R_n:n\in\w\}\in \mathcal R$ and $n\in\w$ let $U_n(R)=\{R\}\cup\bigcup_{m\geq n}R_m^*$. The topology $\tau$ is generated by the base $\B^*=\{B^*:B\in\B\}\cup\{U_n(R):R\in\mathcal R,n\in\w\}$. By $X(\mathcal R)$ we denote the space $(X\cup \mathcal R,\tau)$. It is easy to check that the space $X(\mathcal R)$ is Hausdorff and zero-dimensional. 

The following result is proved by Bell in~\cite[Claim 3.2]{Bell}.

\begin{proposition}\label{Bell}
Let $X$ be a first-countable zero-dimensional space and $\B$ be a base of $X$ consisting of clopen sets.  Suppose that for every $T\in \D(\B)$ there exists $S\in\E(\B)$ such that $S$ refines $T$. Then for every
maximal eventually disjoint subfamily $\mathcal R$ of $\E(\B)$ the space $X(\mathcal R)$ is zero-dimensional, first countable, and pseudocompact.   
\end{proposition}

\begin{proposition}\label{pseudocompact}
Each first-countable zero-dimensional Hausdorff space $X$ of weight $<\min\{\mathfrak{s},\mathfrak{b}\}$ embeds densely into a first-countable zero-dimensional pseudocompact space.    
\end{proposition}

\begin{proof}
Fix a base $\B$ of $X$ of size $<\min\{\mathfrak{s},\mathfrak{b}\}$ consisting of clopen sets. We lose no generality assuming that $\B=\{U_n^x:x\in X, n\in\w\}$, where for every $x\in X$ the family $\{U_n^x:n\in\w\}$ forms a nested open neighborhood base at $x$. % where $U_0^x=X$ for every $x\in X$. 
By Proposition~\ref{Bell} it is enough to check that for every $T\in \D(\B)$ there exists $S\in\E(\B)$ such that $S$ refines $T$. Consider an arbitrary sequence $T=\{T_n:n\in\w\}\in \D(\B)$ and for each $n\in\w$ fix a point $x_n\in T_n$. To each $B\in\B$ we assign a function $f_B\in\w^\w$ defined as follows: if $x_n\in B$, then put $f_B(n)=\min\{k\in \w: U_k^{x_n}\subset B\cap T_n\}$; otherwise put $f_B(n)=\min\{k\in \w: U_k^{x_n}\subset T_n$ and $U_k^{x_n}\cap B=\emptyset\}$. Since $|\B|<\mathfrak b$, there exists a function $f\in\w^\w$ such that $f\geq^* f_B$ for every $B\in\B$. Consider the refinement $\{U^{x_n}_{f(n)}:n\in\w\}$ of $T$. Clearly, for each $B\in \B$ there exists $k\in\w$ such that for every $n\geq k$ either $U^{x_n}_{f(n)}\cap B=\emptyset$ or $U^{x_n}_{f(n)}\subset B$. So, to each $B\in\B$ we can assign a set $C_B=\{n\in\w: U^{x_n}_{f(n)}\subset B\}\subset\w$. Since $|\B|<\mathfrak s$ we get that the family $\{C_B:B\in\B\}$ is not splitting. Hence there exists an infinite subset $A$ of $\w$ possessing the following property: for every $B\in \B$ there exists $k\in\w$ such that either $U^{x_n}_{f(n)}\subset B$ for each $n\in A\setminus k$  or  $U^{x_n}_{f(n)}\cap B=\emptyset$ for each $n\in A\setminus k$. Hence the sequence $S=\{U^{x_n}_{f(n)}:n\in A\}$ is a $\B$-exact refinement of $T$.      
\end{proof}

\section{Non-embedding theorems}
To each almost disjoint family $\mathcal A\subset [\w]^\w$ corresponds a {\em Mrowka space} $\psi(\mathcal A)=(\mathcal A\cup\w,\tau)$, where the topology $\tau$ is defined as follows: points of $\w$ are isolated and for each $A\in\mathcal A$ the family $\{\{A\}\cup (A\setminus n):n\in\w\}$ is an open neighborhood base at $A$ in $\tau$. We refer the reader to~\cite{hrusak} for basic properties of Mrowka spaces.

\begin{proposition}\label{b}
There exists an almost disjoint family $\mathcal A$ of cardinality $\mathfrak b$ such that the corresponding Mrowka space $\psi(\mathcal A)$ doesn't embed into first-countable Hausdorff countably compact spaces. 
\end{proposition}

\begin{proof}
Fix an unbounded subset $\{b_{\alpha}:\alpha\in\mathfrak b\}\subset \w^\w$ which satisfies the following conditions: 
\begin{enumerate}
    \item the function $b_{\alpha}$ is increasing for every $\alpha\in\mathfrak b$;
    \item $b_\alpha <^* b_{\beta}$ whenever $\alpha\in\beta$.
\end{enumerate}

For $i\in\w$ set $v_i=\{(i,n):n\in\w\}$.  Let $\mathcal A\defeq\{b_{\alpha}:\alpha\in\mathfrak b\}\cup\{v_i:i\in\w\}$. Condition (2) implies that $\mathcal A$ is an almost disjoint family of subsets of $\w{\times}\w$. 
To derive a contradiction, assume that there exists a first-countable Hausdorff countably compact space $X$ which contains the Mrowka space $\psi(\mathcal A)$ as a subspace. %It is easy to check ({\color{blue}maybe you have a reference for this fact?}) that the space $X$ is regular. 
 Find an infinite subset $C$ of $\omega$ such that the sequence $\{v_{n}:n\in C\}$ converges to a point $x\in X$. Fix an open neighborhood base $\{U_n:n\in \w\}$ at $x$. Without loss of generality we can assume that $\{v_k:k\in C\setminus n\}\subset U_n$ for all $n\in\w$. %such that $\overline{U_{n+1}}\subset U_n$ for each $n\in\w$. 
 Observe that for each $n\in \w$ there exists a function $f_n\in \w^C$ such that $\{(k,m): k\in C\setminus n, m\geq f_n(k)\}\subset U_n$. Since the functions $b_{\alpha}$, $\alpha\in\mathfrak b$ are increasing, the family $\{b_\alpha{\restriction}_C
:\alpha\in\mathfrak b\}$ is unbounded. Then for each $n\in\w$ there exists $\alpha_n\in \mathfrak b$ such that $f_n\not\geq^* b_{\alpha_n}{\restriction}_C$. It follows that  for each $n\in\w$ and $\xi\geq \alpha_n$,
$$|\{(k,m): k\in C\setminus n, m\geq f_n(k)\}\cap b_{\xi}|=\omega.$$ Thus for $\mu= \sup\{\alpha_n:n\in \w\}$ the set $b_{\mu}\cap \{(k,m): k\in C\setminus n, m\geq f_n(k)\}$ is infinite for all $n\in\w$.  Hence $b_{\mu}\in \overline{U_n}$ for all $n\in \w$, which contradicts the Hausdorffness  of $X$.
\end{proof}

The following result is essentially based on the work of van Douwen and Przymusi\'nski~\cite{DP}. Throughout its proof $\mathbb N=\w\setminus \{0\}$. 

\begin{proposition}\label{De}
There exists a separable first-countable normal Lindel\"of space of cardinality $\mathfrak s$ which cannot be densely embedded into a first-countable regular feebly compact space.  
\end{proposition}

\begin{proof}
By $\mathbb Q$ we denote the set of rational numbers from the interval $[0,1]$.
%Fix an arbitrary splitting family $S\subset[\w]^\w$ of cardinality $\mathbb s$ and consider an injective function $f: S\rightarrow [0,1]$
Enumerate the set of all subsets of $\w$ as $\{A_x:x\in [0,1]\}$, and for each $x\in[0,1]$ and $m\in\N$ choose  $q_x(m)\in \mathbb Q$ such that $0<|x-q_x(m)|<1/m$ and put $Q_x=\{q_x(m):m\in\N\}$.
Let $X$ be the set $[0,1]{\times}\{0\}\cup \mathbb Q{\times}\mathbb{\N}$ endowed with a topology $\tau$ satisfying the following conditions:
\begin{itemize}
    \item $\mathbb Q{\times}\mathbb{\N}$ is the set of isolated points;
    \item if $x\in \mathbb Q$, then basic neighborhoods of a point $(x,0)\in[0,1]{\times}\{0\}$ have the form $$B_{m}(x)=\{(a,b)\in X: |x-a|<1/m\}\setminus (Q_x{\times}A_x\cup \{x\}{\times}\mathbb N) \qquad \hbox{for} \quad m\in\mathbb N;$$
    \item if $x\in [0,1]\setminus \mathbb Q$, then basic neighborhoods of a point $(x,0)\in[0,1]{\times}\{0\}$ have the form $$B_{m}(x)=\{(a,b)\in X: |x-a|<1/m\}\setminus (Q_x{\times}A_x) \qquad \hbox{for} \quad m\in\mathbb N.$$
\end{itemize}
Clearly, the space\footnote{This space was originally constructed by van Douwen and Przymusi\'nski~\cite{DP} in order to investigate remainders of \v{C}ech-Stone compactifications.} $X$ is regular, separable and first-countable.
%Taking into account that the subspace $[0,1]{\times}\{0\}$ of $X$ is compact. 

Fix an arbitrary splitting family $\mathcal S\subset[\w]^\w$ of cardinality $\mathfrak s$ and consider the subspace $$Y=\{x\in[0,1]: A_x\in \mathcal S\}{\times}\{0\}\cup\mathbb Q{\times}\mathbb N$$ of $X$. It is clear that $Y$ is regular and first-countable. Taking into account that the subspace $([0,1]{\times}\{0\})\cap Y$ of $Y$ is second-countable, we deduce that $Y$ is separable, Lindel\"of and subsequently normal.

Let us show that $Y$ doesn't embed densely into a first-countable regular feebly compact space. Assuming the contrary, fix any first-countable regular feebly compact space $Z$ which contains $Y$ as a dense subspace.
The following claim is crucial.

\begin{claim}\label{crucial}
For each $B\in[\mathbb N]^{\w}$ there exists $q\in\mathbb Q$ such that $\{q\}{\times}B$ is not a convergent sequence in $Z$.
\end{claim}

\begin{proof}
 Fix an arbitrary infinite subset $B\subseteq\mathbb N$ and find $S\in \mathcal S$ that splits $B$. Let $x\in[0,1]$ be such that $A_x=S$. Fix an arbitrary open neighborhood $W$ of $(x,0)$ in $Z$ such that $W\cap Y\subseteq B_2(x)$. By the regularity of $Z$ there exists an open neighborhood $V$ of $(x,0)$ in $Z$ such that $\cl_Z(V)\subseteq W$. Find $m\in\mathbb N$ such that $B_m(x)\subset V$. Assume that the set $\{q_{x}(m)\}{\times} B$ converges to some $b\in Z$. Then $b\in \cl_Z(V)\subseteq W$. Recall that $W\cap Y\subseteq B_2(x)$ and  $B_2(x)\cap (\{q_x(m)\}{\times}\N)= \{q_x(m)\}{\times}(\N\setminus S)$. Since the set $B\cap S$ is infinite, the open set $W$ contains no open neighborhood of $b$ which contradicts the choice of $W$.
\end{proof}

%For each $q\in \mathbb Q$ let $T_q=\cl_{Z}(\{q\}{\times}\N)$. 
Fix any enumeration $\mathbb Q=\{q_n:n\in\w\}$. Since the set $Q_0\defeq \{q_0\}{\times}\N$ is open and discrete in a feebly compact space $Z$, there exists an accumulation point $z_0$ of $Q_0$. Since $Z$ is first-countable, there exists $A_0\subseteq \N$ such that the sequence $\{(q_0,n):n\in A_0\}$ converges to $z_0$.
Assume that for each $i\leq n$ we constructed a subset $A_i$ of $\N$ satisfying:
\begin{enumerate}
    \item $A_i\subset A_j$ for every $i\leq j$;
    \item the sequence $\{(q_i,n):n\in A_i\}$ converges to some $z_i\in Z$.
\end{enumerate}
Consider the discrete open set $Q_{n+1}\defeq \{q_{n+1}\}{\times}A_n$. Since $Z$ is first-countable and feebly compact, there exists an infinite subset $A_{n+1}\subset A_n$ such that the sequence $\{(q_{n+1},n):n\in A_{n+1}\}$ converges to some point $z_{n+1}\in Z$. 

After completing the induction we obtain a decreasing chain $\{A_i:i\in\w\}$ of infinite subsets of $\N$ and a sequence of points $\{z_i:i\in\w\}$ such that $z_i$ is the limit of the sequence $\{(q_i,n):n\in A_i\}$. Let $A$ be any pseudointersection of the family $\{A_i:i\in\w\}$, i.e. $A\subseteq^* A_i$ for all $i\in\w$. It is easy to see, that the sequence $\{(q_i,n):n\in A\}$ converges to $z_i$ for each $i\in\w$. But this contradicts Claim~\ref{crucial}. 
\end{proof}

The following trivial fact will be useful in the next proposition:

\begin{lemma}\label{trivial}
Suppose that $A\cup B$ is a $G_{\delta}$ subset of a Hausdorff space $X$. If $B$ is countable, then $A$ is $G_{\delta}$. 
\end{lemma}

\begin{proof}
Let $B\setminus A=\{b_n:n\in\w\}$ and $A\cup B=\bigcap_{n\in\w}U_n$, where $U_n$ is open for each $n\in\w$. Then the set $V_n\defeq U_n\setminus \{b_n\}$ is open for each $n\in\w$, and $\bigcap_{n\in\w} V_n=A$, i.e $A$ is $G_{\delta}$.
\end{proof}

%Recall that a subspace $A$ of the Cantor space is called a {\em $\lambda$-set}  if each countable subset of $A$ is $G_{\delta}$.

For each $i\in\w$ let $\w^{\leq i}=\bigcup_{n\leq i}\w^n$ and $\w^{<i}=\bigcup_{n<i}\w^n$. Clearly, the sets $\w^{\leq i}$ and $\w^{<i}$ ordered by the inclusion are subtrees of the complete $\w$-ary tree.
Recall that a subspace $A$ of the Cantor space is called a {\em $\lambda$-set}  if each countable subset of $A$ is $G_{\delta}$. For the definition of a Pixley-Roy hyperspace see Section~\ref{prelim}.
The proof of the next result uses the ideas of Bell~\cite{Bell}.
\begin{proposition}\label{bpseudo}
Let $X$ be a subspace of the Cantor space which is not a $\lambda$-set. Then the Pixley-Roy hyperspace $\PR(X)$ is a zero-dimensional Tychonoff first-countable space that doesn't embed densely into first-countable regular feebly-compact spaces.     
\end{proposition}

\begin{proof}
%Let $X$ be a subspace of the Cantor space which contains a countable subset $A$ that is not $G_{\delta}$. 
Let $A$ be a countable infinite subset of $X$ which is not $G_\delta$.  
Seeking a contradiction, assume that $\PR(X)$ embeds densely into a regular first-countable feebly compact space $Z$. %Note that if we put $Y=\{x\in X: x$ is not isolated$\}$, then $PR(Y)$ is an open subset 
Fix any countable dense subset $A'$ of $X$. Lemma~\ref{trivial} implies that $Q=A\cup A'$ is not $G_{\delta}$. %Fix an enumeration $Q=\{q_i:i\in\w\}$ and 
Fix an increasing sequence $\{F^{\emptyset}_k:k\in\w\}$ of finite subsets of $Q$ such that $\bigcup_{k\in\w}F^{\emptyset}_k=Q$. Let $B_{\emptyset}\defeq X$.  Consider the family $\{\cl_Z[F_k^{\emptyset}, X]: k\in\w\}$. Taking into account that $\PR(X)$ is dense in $Z$ and $[F_k^{\emptyset}, B_{\emptyset}]$ is clopen in $\PR(X)$, it is easy to check that $[F_k^{\emptyset}, B_{\emptyset}]\subset \Int_Z(\cl_Z[F_k^{\emptyset}, B_{\emptyset}])$ for all $k\in\w$. So, the family $\{\Int_Z(\cl_Z[F_k^{\emptyset}, B_{\emptyset}]):k\in\w\}$ consists of open subsets of $Z$ and is centered. By feeble compactness of $Z$, there exists 
$$z_\emptyset\in \bigcap_{k\in\w}\cl_Z(\Int_Z(\cl_Z[F_k^{\emptyset}, B_{\emptyset}]))=\bigcap_{k\in\w}\cl_Z[F_k^{\emptyset}, B_{\emptyset}].$$ 
Fix a countable open neighborhood base $\{O^{\emptyset}_k:k\in\w\}$ at $z_{\emptyset}$. Since $\PR(X)$ is dense in $Z$, for each $k\in\w$ there exists a basic clopen subset $[G_k^{\emptyset}, W_k^{\emptyset}]$ of $\PR(X)$ such that $[G_k^{\emptyset}, W_k^{\emptyset}]\subseteq O_k^{\emptyset}\cap [F_k^{\emptyset}, B_{\emptyset}]$.
Clearly, $F_{k}^{\emptyset}\subseteq G_{k}^{\emptyset}$ and the sequence $\{G_{k}^{\emptyset}:k\in\w\}\subset \PR(X)$ converges to $z_{\emptyset}$. 

Suppose that for some $i\geq 0$ we accomplished the following things:
\begin{enumerate}
\item constructed a subtree $\mathbb T_i$ of $\w^{\leq i}$;
\item constructed a family $\{B_s:s\in \mathbb T_i\}$ of clopen subsets of $X$ satisfying the following conditions:
\begin{itemize}
\item[(2.1)] for every $s\in \mathbb T_i\cap \w^{<i}$, $Q_s\defeq Q\cap B_s\subseteq \bigcup_{n\in \beta_s}B_{s\hat{~}n}\subseteq B_s$, where $\beta_s\defeq \{n\in\w: s\hat{~}n\in\mathbb T_i\}\in\w+1$;
\item[(2.2)] $B_{s\hat{~}n}\cap B_{s\hat{~}m}=\emptyset$ for every $s\in  \mathbb T_i\cap \w^{<i}$ and distinct $n,m\in \beta_s$;
\item[(2.3)] for each $s\in\mathbb T_i$ the diameter of $B_s$ doesn't exceed $1/|s|$. 
\end{itemize}
\item $\forall s\in \mathbb T_i$ fixed an increasing sequence $\{F_k^s:k\in\w\}$ of finite subsets of $Q_s\defeq B_s\cap Q$ such that $\bigcup_{k\in\w}F_k^s= Q_s$;
\item $\forall s\in \mathbb T_i$ fixed a point $z_s\in \bigcap_{k\in\w}\cl_Z[F_k^s,B_s]$ and a nested open neighborhood base $\{O_k^s:k\in\w\}$ at $z_s$;
\item $\forall s\in \mathbb T_i$, $\forall k\in\w$ fixed a subset $[G^s_k,W_k^s]\subset O_k^s\cap [F_k^s,B_s]$ such that the following condition holds:
\begin{itemize}
\item[(5.1)] $\forall s\in\w^{<i}\cap \mathbb T_i$, $\forall n\in \beta_s$, $\exists k(s,n)\geq |s|+1$ such that $B_{s\hat{~}n}\subseteq \bigcap_{j\leq |s|} W_{k(s,n)}^{s{\restriction}j}$.
\end{itemize}
\end{enumerate}

Fix any $s\in \mathbb T_i\cap \w^{i}$. Since $B_s$ is open in $X$ and $Q$ is dense, the set $Q_s\defeq Q\cap B_s$ is not empty. Note that since $X$ could contain isolated points, it is possible that $|Q_s|<\w$. So, let us fix an enumeration $Q_s=\{q_n^s:n\in \alpha_s\}$, where $\alpha_s\in\w+1$. Observe that $q^s_0\in Q\cap B_{s{\restriction}j}$ for all $j\leq |s|$. Since  $\{F_k^{s{\restriction}j}: k\in\w\}$ is an increasing sequence of finite subsets of $Q\cap B_{s{\restriction}j}$, there exists $n_j\in\w$ such that $q^s_0\in F_{n}^{s{\restriction}j}$ for all $n\geq n_j$. Put $k(s,0)=\max\{n_j:j\leq |s|\}+|s|+1$ and note that the open set $\bigcap_{j\leq |s|} W_{k(s,0)}^{s{\restriction}j}$ contains $q^s_0$. So fix an arbitrary clopen set $B_{s\hat{~}0}\subseteq \bigcap_{j\leq |s|} W_{k(s,0)}^{s{\restriction}j}$ which contains $q_0^s$ and has diameter $\leq 1/(i+1)$. By condition (5), we get that $W_{k(s,0)}^s\subseteq B_{s}$ which implies that $B_{s\hat{~}0}\subseteq B_s$. If $Q_s\subset B_{s\hat{~}0}$, we put $\beta_s=1$ and move to another $s\in  \mathbb T_i\cap \w^{i}$. If $Q_s\setminus B_{s\hat{~}0}\neq \emptyset$, find minimal $m\in\w$ such that $q_m^s\notin B_{s\hat{~}0}$. Similarly as above there exists $k(s,1)\geq |s|+1$ such that the open set $\bigcap_{j\leq |s|} W_{k(s,1)}^{s{\restriction}j}$ contains $q^s_m$. Since the set $B_{s\hat{~}0}$ is clopen and doesn't contain $q_m^s$, we can find a clopen neighborhood $B_{s\hat{~}1}$ of $q_m^s$ of diameter $\leq 1/(i+1)$ which is contained in $(B_s\setminus B_{s\hat{~}0})\cap \bigcap_{j\leq |s|} W_{k(s,1)}^{s{\restriction}j}$. %Clearly, $B_{s\hat{~}1}\subseteq B_s\setminus B_{s\hat{~}0}$.
If $Q_s\subset B_{s\hat{~}0}\cup B_{s\hat{~}1}$, we put $\beta_s=2$ and move to another $s\in \mathbb T_i\cap \w^i$. Otherwise, we can find minimal $m\in\w$ such that $q_m^s\in B_s\setminus (B_{s\hat{~}0}\cup B_{s\hat{~}1})$ and repeat the previous steps. Proceeding this way we shall construct
a family $\{B_{s\hat{~}k}:k\in\beta_s\in\w+1\}$ of clopen subsets of  $X$ which satisfies conditions (2.1), (2.2), (2.3) and (5.1). Note that $\beta_s\leq \alpha_s$ can be finite. Let $\mathbb T^s_{i+1}=\mathbb T_i\cup \{s\hat{~}k:k\in\beta_s\}$.

After making the above construction for all $s\in \mathbb T_i\cap \w^{i}$, we obtain a tree $\mathbb T_{i+1}=\bigcup_{s\in  \mathbb T_i\cap \w^{i}}\mathbb T_{i+1}^s$ and a family $\{B_{s\hat{~}k}:k\in\beta_s \hbox{ and } s\in  \mathbb T_i\cap \w^{i}\}$ of clopen subsets of  $X$ which satisfies conditions (2.1), (2.2),  (2.3) and (5.1).

For each $s\in \mathbb T_{i+1}\cap \w^{i+1}$ fix an increasing sequence $\{F_k^s:k\in\w\}$ of finite subsets of $Q_s$ such that $\bigcup_{k\in\w}F_k^s=Q_s$. Analogous arguments to those we used to find $z_\emptyset$ imply that for each $s\in \mathbb T_{i+1}\cap \w^{i+1}$ there exists a point $z_s\in \bigcap_{k\in\w}\cl_Z[F_k^s,B_s]$. By the first-countability of $Z$, we can fix a nested open neighborhood base $\{O_k^s:k\in\w\}$ at $z_s$. For each $s\in \mathbb T_{i+1}\cap \w^{i+1}$ and $k\in\w$ fix a  subset $[G^s_k,W_k^s]\subset O_k^s\cap [F_k^s,B_s]$. This finishes step $i+1$ of the induction. 

Upon completing the induction we obtain a tree $\mathbb T=\bigcup_{i\in\w}\mathbb T_i$, and a family $\{B_s:s\in\mathbb T\}$ of clopen subsets of $X$ which satisfies conditions (2.1), (2.2), (2.3) and (5.1). Also, for each $s\in\mathbb T$ we found an element $z_s\in Z$ which satisfies condition (4) and constructed the families $\{W_k^s: k\in\w, s\in \mathbb T\}$ (consisting of open subsets of $X$) and $\{G_k^s: k\in\w, s\in \mathbb T\}$ (consisting of finite subsets of $X$), which satisfy conditions (5) and (5.1).

Let $$P=\bigcap_{i\in\w}\bigcup_{s\in\w^i\cap \mathbb T}B_s.$$ By the construction, $P$ is a $G_{\delta}$-set containing $Q$. It follows that  $|P|>\w$. Indeed, otherwise we would have that 
$Q=P\cap\bigcap_{x\in P\setminus Q}(X\setminus\{x\})$ is a $G_\delta$ subset of $X$, which is not the case.
Thus, we can fix  
\begin{equation}\label{eqqq}
 p_{\infty}\in P\setminus \bigcup_{s\in \mathbb T}\bigcup_{k\in\w}G_k^s.   
\end{equation} 
It is easy to see that there exists $t\in\w^\w$ such that $p_{\infty}\in B_{t{\restriction}i}$ for any $i\in\w$. Note that since the diameter of $B_{t{\restriction}i}$ doesn't exceed $1/i$, we get that the family $\{B_{t{\restriction}i}:i\in\w\}$ forms an open neighborhood base at $p_{\infty}$ in $X$. It follows that the family $\{[\{p_{\infty}\},B_{t{\restriction}i}]:i\in\w\}$ forms an open neighborhood base of $\{p_{\infty}\}$ in $\PR(X)$.  Put $R=\Int_Z(\cl_Z([\{p_{\infty}\}, X]))$. One can check that $R\cap \PR(X)=[\{p_{\infty}\}, X]$. By the regularity of $Z$, there exists an open in $Z$ neighborhood $W$ of $\{p_{\infty}\}$ such that $\cl_Z(W)\subseteq R$. Thus, there exists $i\in\w$ such that  $\cl_Z([\{p_{\infty}\}, B_{t{\restriction}i}])\subseteq R$.

\begin{claim}\label{z}
$z_{t{\restriction}i}\in\cl_Z([\{p_{\infty}\}, B_{t{\restriction}i}])$.    
\end{claim}

\begin{proof}
Pick any integer $l>i$. Using condition (5.1) for $s=t{\restriction}l$ and $n=t(l)$ we can find an integer $k(t{\restriction}l,t(l))\geq l+1$ such that $B_{t{\restriction}(l+1)}\subseteq W_{k(t{\restriction}l,t(l))}^{t{\restriction}i}$.  Then $p_{\infty}\in  W_k^{t{\restriction}i}$ for infinitely many $k\in\w$. It follows that $[\{p_{\infty}\}, B_{t{\restriction}i}]\cap [G_k^{t{\restriction}i}, W_k^{t{\restriction}i}]\neq \emptyset$ for infinitely many $k$. Since for each $k\in\w$, $[G_k^{t{\restriction}i},W_k^{t{\restriction}i}]\subset O_k^{t{\restriction}i}$ and the family $\{O_k^{t{\restriction}i}: k\in\w\}$ forms an open neighborhood base at $z_{t{\restriction}i}$, we deduce that $z_{t{\restriction}i}\in\cl_Z([\{p_{\infty}\}, B_{t{\restriction}i}])$.   
\end{proof}

Since $\cl_Z([\{p_{\infty}\}, B_{t{\restriction}i}])\subseteq R$, Claim~\ref{z} implies that $z_{t{\restriction}i}\in R$.
Condition (5) implies that for each $s\in \mathbb T$ the sequence $\{G_k^s:k\in\w\}\subseteq \PR(X)$ converges to $z_s$. Since the set $R$ is an open neighborhood of $z_{t{\restriction}i}$, there exists $k\in\w$ such that $G_k^{t{\restriction}i}\in R\cap \PR(X)=[\{p_{\infty}\}, X]$, which is impossible, as $p_\infty\notin  \bigcup_{s\in \mathbb T}\bigcup_{k\in\w}G_k^s$ (see Equation~\ref{eqqq}). 
\end{proof}

By \cite[Corollary 8.51]{Buk} there exists a subspace $A$ of the Cantor space which is not a $\lambda$-set and has cardinality $\mathfrak b$. Note that $|\PR(A)|=|A|=\mathfrak b$. Hence Proposition~\ref{bpseudo} implies the following: 

\begin{corollary}\label{pseudob}
There exists a Tychonoff zero-dimensional first-countable space $X$ of cardinality $\mathfrak b$ which doesn't embed densely into regular first-countable  feebly compact spaces.   
\end{corollary}

%\begin{corollary}
%A subspace $X$ of the Cantor space is a $\lambda$-set if and only if the Pixley-Roy hyperspace $\PR(X)$ embeds densely into a first-countable pseudocompact space.      
%\end{corollary}

%The following result was proved in~\cite{Ny1}.
%\begin{theorem} \label{old} 
%Every locally compact, locally countable space of weight $< \mathfrak c$ can be (densely) embedded in a locally countable, countably compact space if and only if $\mathfrak b = \mathfrak c$. 
%  \end{theorem} 

\section{Proofs of the main results}

\subsection*{Proof of Theorem~\ref{comp}}

We need to show that the following assertions are equivalent:
\begin{enumerate}
    \item $\w_1=\mathfrak c$.
    \item Every first-countable Tychonoff space of weight $<\mathfrak c$ embeds into a Hausdorff first-countable compact space.
    \item Each separable first-countable locally compact normal space of cardinality $<\mathfrak c$ embeds into a Hausdorff first-countable compact space.
\end{enumerate}

\begin{proof}
Implication (1) $\Rightarrow$ (2) follows from the fact that Each Tychonoff second-countable space is a subspace of the Tychonoff cube $[0,1]^{\w}$. Implication (2) $\Rightarrow$ (3) follows from the fact that the weight of a first-countable space doesn't exceed its cardinality. Proposition~2.11 from~\cite{BSZ} implies the existence of a separable normal locally compact first-countable space $X$ which contains a homeomorphic copy of the cardinal $\w_1$ endowed with the order topology. Since the only compactification of $\w_1$ is $\w_1{+}1$, which is not first-countable, we deduce that $X$ cannot be embedded into a Hausdorff first-countable compact space.     
\end{proof}

\subsection*{Proof of Theorem~\ref{lc}}
We need to show that the following statements are equivalent.   
\begin{enumerate}	
\item $\mathfrak b = \mathfrak c$.
\item Every Hausdorff, locally compact, first-countable space of weight $< \mathfrak c$ can be (densely) embedded in a Hausdorff first-countable locally compact countably compact space.
 \item Every Hausdorff, locally compact, first-countable space of cardinality $< \mathfrak c$ can be (densely) embedded in a Hausdorff first-countable  countably compact space.
\end{enumerate}
  \begin{proof}
The implication (1) $\Rightarrow$ (2) follows from Proposition~\ref{lce}.
Since the weight of an infinite first-countable space doesn't exceed its cardinality, implication (2) $\Rightarrow$ (3) follows. Taking into account that Mrowka spaces are first-countable and locally compact, implication (3) $\Rightarrow$ (1) follows from Proposition~\ref{b}.  
\end{proof}

\subsection*{Proof of Theorem~\ref{sbc}}
We need to show that the following assertions are equivalent.   
\begin{enumerate}	
\item $\mathfrak b = \mathfrak s= \mathfrak c$.
%\item Every Hausdorff zero-dimensional first-countable space of weight $< \mathfrak c$ can be densely embedded in a Hausdorff zero-dimensional first-countable countably compact space.
%\item Every Hausdorff zero-dimensional first-countable space of cardinality $< \mathfrak c$ can be densely embedded in a Hausdorff zero-dimensional first-countable countably compact space.
 \item Every regular first-countable space of weight $< \mathfrak c$ can be densely embedded in a regular first-countable countably compact space.
  \item Every regular first-countable space of cardinality $< \mathfrak c$ can be densely embedded in a regular first-countable countably compact space.
 \end{enumerate}

  \begin{proof}
Implication (1) $\Rightarrow$ (2) follows from Proposition~\ref{mainr}. Since the weight of an infinite first-countable space doesn't exceed its cardinality, implication (2) $\Rightarrow$ (3) holds true. 

(3) $\Rightarrow$ (1). Proposition~\ref{b} implies that $\mathfrak b=\mathfrak c$. Proposition~\ref{De} implies that $\mathfrak s=\mathfrak c$. %Hence assertion (1) holds.
\end{proof}

\subsection*{Proof of Theorem~\ref{sbc1}}
We need to show that the following assertions are equivalent.   
\begin{enumerate}	
\item $\mathfrak b = \mathfrak s= \mathfrak c$.
\item Every Hausdorff zero-dimensional first-countable space of weight $< \mathfrak c$ can be densely embedded into a Hausdorff zero-dimensional first-countable countably compact space.
\item Every Hausdorff zero-dimensional first-countable space of cardinality $< \mathfrak c$ can be densely embedded into a Hausdorff zero-dimensional first-countable countably compact space.
 \end{enumerate}

  \begin{proof}
Implication (1) $\Rightarrow$ (2) follows from Proposition~\ref{mainz}. Since the weight of an infinite first-countable space doesn't exceed its cardinality, implication (2) $\Rightarrow$ (3) holds true. 

(3) $\Rightarrow$ (1) Proposition~\ref{b} implies that $\mathfrak b=\mathfrak c$.
Recall that a Tychonoff first-countable space $Y$ constructed in the proof of Proposition~\ref{De} has cardinality $\mathfrak s$. If $\mathfrak s<\mathfrak c$, then the space $Y$ is zero-dimensional (as every Tychonoff space of cardinality $<\mathfrak c$ is zero-dimensional), and $Y$ doesn't embed  into regular first-countable countably compact spaces, which contradicts assertion (3). Hence $\mathfrak s=\mathfrak c$, i.e. assertion (1) holds.   \end{proof}  

In order to prove Theorem~\ref{rigid} we need the following auxiliary lemma.

\begin{lemma}\label{nonnormal}
 There exists a regular zero-dimensional separable first-countable non-normal space of cardinality $\w_1$.   
\end{lemma}

\begin{proof}
We shall use ideas of Nyikos and Vaughan~\cite{NV}. Consider a Hausdorff $(\omega_1,\omega_1)$ gap, i.e. a family $\{A_{\alpha}:\alpha\in\w_1\}\cup \{B_{\alpha}:\alpha\in\w_1\}\subset [\w]^{\w}$ which satisfies the following conditions:
\begin{itemize}
    \item $A_{\alpha}\subset^* B_{\beta}$ for any $\alpha,\beta\in\w_1$;
    \item $A_{\alpha}\subset^* A_{\beta}$ for any $\alpha<\beta<\w_1$;
    \item $B_{\alpha}\subset^* B_{\beta}$ for any $\beta<\alpha<\w_1$;
    \item there exists no $C\in[\w]^\w$ such that $A_{\alpha}\subset^* C\subset^* B_{\alpha}$ for all $\alpha\in\w_1$.
\end{itemize}
For a finite set $F$ and ordinals $\alpha<\beta<\w_1$ let us introduce two sorts of sets:
$$A(\alpha,\beta,F)=\{A_{\xi}\mid \alpha<\xi\leq \beta\}\cup ((A_{\beta}\setminus A_{\alpha})\setminus F)\quad\hbox{and}$$
$$B(\alpha,\beta,F)=\{B_{\xi}\mid \alpha<\xi\leq \beta\}\cup ((B_{\alpha}\setminus B_{\beta})\setminus F).$$
Let $N$ be the set $\{A_{\alpha}:\alpha\in\w_1\}\cup \{B_{\alpha}:\alpha\in\w_1\}\cup\w$ endowed with the topology $\tau$ defined by the subbase: 
\begin{align*}
\mathcal B=\{\{n\}:n\in\w\}\cup\{A_0\}\cup\{B_0\}\cup\{A(\alpha,\beta,F):\alpha<\beta<\w_1, F\in[\w]^{<\w}\}\cup\\
\cup\{B(\alpha,\beta,F):\alpha<\beta<\w_1, F\in[\w]^{<\w}\}.   
\end{align*}

It is straightforward to check that the space $N$ is regular, separable, first-countable and has cardinality $\w_1$. Also, the closed sets $\{A_{\alpha}:\alpha\in\w_1\}$ and $\{B_{\alpha}:\alpha\in\w_1\}$ cannot be separated by disjoint open sets. Hence the space $N$ is not normal.   
\end{proof}

\subsection*{Proof of Theorem~\ref{rigid}}
We need to show the equivalence of the following assertions:
\begin{enumerate}
    \item $\mathfrak b = \mathfrak s= \mathfrak c$.
    \item Every regular separable first-countable non-normal space of weight $< \mathfrak c$ embeds into a Nyikos $\mathbb R$-rigid space.
     \item Every regular separable first-countable non-normal space of cardinality $< \mathfrak c$ embeds into a Nyikos $\mathbb R$-rigid space.
\end{enumerate}    

\begin{proof}
(1) $\Rightarrow$ (2). Let $X$ be a regular separable first countable non-normal space of weight $<\mathfrak c$.
Applying Jones machine (see~\cite{BZ} or~\cite{J} for details) we embed $X$ into a first-countable regular separable space $J(X)$ of weight $<\mathfrak c$ containing two points $a,b$  which cannot be separated by any real-valued continuous function. Applying the modified van Douwen's extension $E_{a,b}(J(X))$ of $J(X)$ (see Section 5 in~\cite{BZ}) we embed $J(X)$ into a regular separable first-countable $\R$-rigid space of weight $<\mathfrak c$. By Theorem~\ref{mainr} we get a regular first-countable countably compact space $Z$ which contains the separable $\mathbb R$-rigid space $E_{a,b}(J(X))$ as a dense subspace. Hence $Z$ is an $\R$-rigid Nyikos space which contains a homeomorphic copy of $X$.

The implication (2) $\Rightarrow$ (3) holds true, because  the weight of an infinite first-countable space doesn't exceed its cardinality.

(3) $\Rightarrow$ (1). Let $N_1$ be the topological sum of the space $N$ constructed in Lemma~\ref{nonnormal} and the Mrowka space $\psi(\mathcal A)$ constructed in Proposition~\ref{b}. It is clear that $N_1$ is a regular separable first-countable non-normal space of cardinality $\mathfrak b$ which doesn't embed into Nyikos spaces. This example implies that $\mathfrak b=\mathfrak c$. 

Let $N_2$ be the topological sum of the space $N$ constructed in Lemma~\ref{nonnormal} and the space $Y$ constructed in Proposition~\ref{De}. It can be checked that $N_2$ is a separable regular first-countable non-normal space of cardinality $\mathfrak s$ which doesn't embed into Nyikos spaces. Hence $\mathfrak s=\mathfrak c$.  
\end{proof}

%Note that the latter theorem vacuously holds if $\mathfrak c=\omega_1$, as each second countable regular space is normal.

\subsection*{Proof of Theorem~\ref{BZZnew}}
Assuming that $\w_1<\mathfrak b=\mathfrak s=\mathfrak c$ we need to construct an $\mathbb R$-rigid Nyikos space. 

\begin{proof}
Consider the space $N$ constructed in Lemma~\ref{nonnormal} and apply item (3) of Theorem~\ref{rigid} to it.
\end{proof}

\subsection*{Proof of Corollary~\ref{cormain}}
We need to show that under PFA the following assertions hold:
\begin{enumerate}
 \item Every normal Nyikos space is compact. 
 \item Every regular separable first-countable space of weight $< \mathfrak c$ embeds into an $\mathbb R$-rigid Nyikos space.
\end{enumerate}
\begin{proof}
Assertion (1) follows from Theorem~\ref{NZ}. In order to prove assertion (2) fix any regular separable first-countable space $X$ of weight $< \mathfrak c$. Let $N$ be the space constructed in Lemma~\ref{nonnormal}. Recall that PFA implies $\w_1<\mathfrak b=\mathfrak s=\mathfrak c$. Hence the topological sum $X\cup N$ is a regular separable first-countable non-normal space of weight $<\mathfrak c$. By Theorem~\ref{rigid}(2), $X\cup N$ embeds into an $\mathbb R$-rigid Nyikos space.
\end{proof}

In order to prove Theorem~\ref{Tychonoff} we shall need the following auxiliary lemma.

\begin{lemma}\label{tech}
Let $X$ be a first-countable Tychonoff space of weight $\kappa$. Then for each $x\in X$ there exists an open neighborhood base $\mathcal B_x=\{U_n^x:n\in\w\}$ at $x$ such that  $U^x_{n+1}\subseteq U^x_n$ for any $n\in\w$, and $\mathcal B=\bigcup_{x\in X}\mathcal B_x$ is a base of $X$ of size $\kappa$ satisfying the following condition:
\begin{itemize}
    \item[(i)] for each $x\in X$ and $U_n^x\in\mathcal B_x$ there exists $f_{U_n^x}\in C(X,[0,1])$ such that $f_{U_n^x}{\restriction}_{X\setminus U_n^x}\equiv 0$ and  $f_{U_n^x}{\restriction}_{U_{n+1}^x}\equiv 1$.
\end{itemize}
\end{lemma}

\begin{proof}
Let $\mathcal D$ be a base of $X$ of size $\kappa$. 
Put
$$\mathcal Y=\{\langle D_0,D_1\rangle\in \mathcal D^2: \exists f_{\langle D_0,D_1\rangle}\in C(X,[0,1])\text{ such that }f_{\langle D_0,D_1\rangle}{\restriction}_{D_1}\equiv 1 \text{ and } f_{\langle D_0,D_1\rangle}{\restriction}_{X\setminus D_0}\equiv 0\}.$$

\begin{claim}\label{help}
For every $x\in X$ and open neighborhood $U$ of $x$ there exists $\langle D_0,D_1\rangle\in\mathcal Y$ such that $x\in D_1\subset D_0\subset U$.    
\end{claim}

\begin{proof}
Fix any $x\in X$ and an open neighborhood $U$ of $x$. Since $\mathcal D$ is a base, there exists an open set $D_0\in\mathcal D$ such that $x\in \mathcal D_0\subseteq U$. Then there exists a continuous function $h\in C(X,[0,1])$ such that $h{\restriction}_{X\setminus D_0}\equiv 0$ and $h(x)=1$. Since the set $h^{-1}([1/2,1])$ contains $x$ in its interior, there exists $D_1\in\mathcal D$ such that $D_1\subset h^{-1}([1/2,1])$. Clearly, $f(z)\defeq 2\min\{h(z),1/2\}\in  C(X,[0,1])$. It is easy to see that $f{\restriction}_{D_1}\equiv 1$ and  $f{\restriction}_{X\setminus D_0}\equiv 0$, i.e. $\langle D_0,D_1\rangle\in \mathcal Y$.  
\end{proof}

Fix any $x\in X$ and an open neighborhood base $\mathcal V=\{V_n:n\in\w\}\subset \mathcal D$ of $x$ such that $V_{n+1}\subset V_n$ for all $n\in\w$. Wlog we can assume that $V_0=X$. Let $U_0^x=X$ and $f_{U_0^x}$ be the constant map from $X$ to $[0,1]$ with value $1$. Claim~\ref{help} implies that there exist a pair $\langle D_0^1,D_1^1\rangle\in\mathcal Y$ such that $x\in D_0^1\subset D_1^1\subseteq V_1$ and a function $f_{\langle D_0^1,D_1^1\rangle}\in C(X,[0,1])$ such that $f_{\langle D_0^1,D_1^1\rangle}{\restriction}_{D_1^1}\equiv 1$ and  $f_{\langle D_0^1,D_1^1\rangle}{\restriction}_{X\setminus D_0^1}\equiv 0$. Put $U_1^x=D_0^1$ and $f_{U_1^x}=f_{\langle D_0^1,D_1^1\rangle}$.
Using Claim~\ref{help} again we can find a pair $\langle D_0^2,D_1^2\rangle\in\mathcal Y$ such that $x\in D_0^2\subset D_1^2\subseteq V_2\cap D_1^1$ and a function $f_{\langle D_0^2,D_1^2\rangle}\in C(X,[0,1])$ such that $f_{\langle D_0^2,D_1^2\rangle}{\restriction}_{D_1^2}\equiv 1$ and  $f_{\langle D_0^2,D_1^2\rangle}{\restriction}_{X\setminus D_0^2}\equiv 0$. Put $U_2^x=D_0^2$ and $f_{U_2^x}=f_{\langle D_0^2,D_1^2\rangle}$. Proceeding this way, for each $x\in X$ we obtain an open neighborhood base $\{U_n^x:n\in\w\}$ at $x$ and a family $\{f_{U_n^x}:U_n^x\in\mathcal B_x\}\subset C(X,[0,1])$ such that $\mathcal B=\bigcup_{x\in X}\mathcal B_x$ satisfies condition (i). Since $\mathcal B\subseteq \mathcal D$ we get that $|\mathcal B|\leq \kappa$. 
\end{proof}

Recall that by ($\heartsuit$) we denote the following consistent assumption: ``$\mathfrak b=\mathfrak c$ and there exists a $P_{\mathfrak c}$-point in $\beta(\w)$''. 

\subsection*{Proof of Theorem~\ref{Tychonoff}}

We need to show that assuming ($\heartsuit$) every Tychonoff first-countable space $X$ of weight $\kappa<\mathfrak{c}$ embeds densely  into a Tychonoff first-countable countably compact space.

\begin{proof}
Let $X$ be a first-countable Tychonoff space such that $w(X)<\mathfrak{c}$. Without loss of generality we can assume that the underlying set of $X$ is disjoint with $\mathfrak c$. The first countability of $X$ implies that $|X|\leq \mathfrak c^{\w}=\mathfrak c$. If $X$ is countably compact, then there is nothing to prove. Otherwise, let 
$$\mathcal{D}=\{A\in [X]^{\omega}: \hbox{ } A \hbox{ is closed and discrete in }X \}.$$ Fix any bijection $h: \mathcal{D}\cup [\mathfrak{c}]^{\omega}\rightarrow \mathfrak{c}$ such that $h(a)\geq \sup(a)$ for any $a\in [\mathfrak{c}]^{\omega}$. %It is easy to see that such a bijection exists. 
Next, for every $\alpha\leq \mathfrak{c}$ we shall recursively construct a topology $\tau_{\alpha}$ on $X_{\alpha}\subseteq X\cup\alpha$. For the sake of brevity we denote the space $(X_{\alpha},\tau_{\alpha})$ by $Y_{\alpha}$. At the end, we will show that the space $Y_{\mathfrak{c}}$ has the desired properties.

Let $X_0=X$. By Lemma~\ref{tech}, there exists a base $\mathcal B_0=\bigcup_{x\in X}\mathcal B_{0}^x$ of size $<\mathfrak c$ of the topology on $X$, where for each $x\in X$, the collection $\mathcal B_{0}^x=\{U_{n,0}^x:n\in \w\}$ is a nested open neighborhood base at $x$ and satisfies the following:  $U_{0,0}^x=X$; for every $n>0$ and $x\in X_0$
 there exists a continuous function $f_{(U_{n,0}^x, U_{n+1,0}^x)}:X\rightarrow [0,1]$ such that $f_{(U_{n,0}^x, U_{n+1,0}^x)}{\restriction}_{X\setminus U_{n,0}^x}\equiv 0$ and $f_{(U_{n,0}^x, U_{n+1,0}^x)}{\restriction}_{U^x_{n+1,0}}\equiv 1$. We also set $f_{(U_{0,0}^x, U_{1,0}^x)}\equiv 1$.

%$U_{n,0}^x\in \mathcal B_0$ there exists a continuous function $f_{U_{n,0}^x}:X\rightarrow [0,1]$ such that $f_{U_{n,0}^x}{\restriction}_{X\setminus U_{n,0}^x}\equiv 0$, and $f_{U_{n,0}^x}{\restriction}_{U^x_{n+1,0}}\equiv 1$. Also put $f_{U_{0,0}^x}\equiv 1$. 
%{\color{red} Why this is possible? It would probably work if we know that $f_{U_{n,0}^x}(x)=1$, but we only know that $f_{U_{n,0}^x}(x)>0$. In theory there could exist $x\in X$ such that $f_{U_{n,0}^x}(x)=1/2$ for all $n\in\w$.}

Assume that for each $\alpha<\xi$ the Tychonoff first-countable spaces $Y_{\alpha}$ are already constructed by defining a base $\mathcal B_{\alpha}=\bigcup_{x\in X_{\alpha}}\mathcal B_{\alpha}^x$ of the the topology $\tau_{\alpha}$ on a set $X_{\alpha}\subseteq X\cup\alpha$, where for each $x\in X_{\alpha}$, the collection $\mathcal B_{\alpha}^x=\{U_{n,\alpha}^x:n\in \w\}$ is a nested open neighborhood base at $x$. Additionally assume that $X_{\alpha}\subset X_{\beta}$ for any $\alpha\in\beta$, and we have in parallel  constructed a family $$\Phi_\alpha=\{f_{(U_{n,\alpha}^x, U_{n+1,\alpha}^x)}:n\in\w, x\in X_\alpha\},$$ where for each $(U_{n,\alpha}^x, U_{n+1,\alpha}^x)$ the function $f_{(U_{n,\alpha}^x, U_{n+1,\alpha}^x)}: Y_{\alpha}\rightarrow [0,1]$ is continuous and $$f_{(U_{n,\alpha}^x, U_{n+1,\alpha}^x)}{\restriction}_{X_\alpha\setminus U_{n,\alpha}^x} \equiv 0.$$
 In what follows, for brevity, we put $E(x,n,\alpha)=(U_{n,\alpha}^x, U_{n+1,\alpha}^x)$ for all $n\in\w$, $\alpha<\xi$ and $x\in X_0$.

The families $\mathcal B_{\alpha}$ and $\Phi_{\alpha}$ are presumed to satisfy the following conditions  for all $\alpha<\xi$:
\begin{enumerate}
    \item $|\mathcal B_{\alpha}|<\mathfrak c$;
    \item $U_{0,\alpha}^x=X_{\alpha}$ for each $x\in X_{\alpha}$;
    \item $f_{E(x,0,\alpha)}\equiv 1$ for any $x\in X_{\alpha}$;
    \item for every $x\in X_{\alpha}$ and $n\in\w$ we have $f_{E(x,n,\alpha)}{\restriction}_{U_{n+1,\alpha}^x}\equiv 1$ and $f_{E(x,n,\alpha)}{\restriction}_{X_\alpha\setminus U_{n,\alpha}^x} \equiv 0$;
    %\item for every $x\in X_{\alpha}$ and $n\in\w$ there exists a continuous function $f_\frac{x}{n,\alpha}:Y_{\alpha}\rightarrow \mathbb R$ such that $f_{\frac{x}{n,\alpha}}^{-1}(0)=X_\alpha\setminus U_{n,\alpha}^x$.
    \item for every $\beta<\alpha$, $n\in\w$ and $x\in X_{\beta}$ we have $U_{n,\beta}^x=U_{n,\alpha}^x\cap X_\beta$ and $\cl_{Y_\alpha}(U_{n,\beta}^x)=\cl_{Y_\alpha}(U_{n,\alpha}^x)$;
    \item for every $\beta<\alpha$, $n\in\w$ and $x\in X_{\beta}$ we have $f_{E(x,n,\alpha)}{\restriction}_{X_{\beta}}=f_{E(x,n,\beta)}$.
\end{enumerate}

There are three cases to consider:
\begin{itemize}
\item[1)] $\xi=\gamma+1$ for some $\gamma\in\mathfrak{c}$ and $h^{-1}(\gamma)\cap X_{\gamma}$ is not an infinite closed discrete subset of $Y_{\gamma}$;
\item[2)] $\xi=\gamma+1$ for some $\gamma\in\mathfrak{c}$ and $h^{-1}(\gamma)\cap X_{\gamma}$ is an infinite closed discrete subset of $Y_{\gamma}$;
\item[3)] $\xi$ is a limit ordinal.
\end{itemize}

\medskip

1) Let $X_{\xi}=X_{\gamma}$, $\mathcal B_{\xi}=\mathcal B_{\gamma}$ and $\Phi_\xi=\Phi_{\gamma}$.
\medskip 

2) Put $X_{\xi}=X_{\gamma}\cup\{\gamma\}$. Let $h^{-1}(\gamma)=\{z_i\}_{i\in\omega}$. Fix any $P_{\mathfrak{c}}$-point $p$ which exists by the assumption. 
Since the set $h^{-1}(\gamma)$ is closed and discrete, for any $x\in X_{\gamma}$ there exists a positive integer $k(x)$ such that  $|U_{k(x),\gamma}^{x}\cap h^{-1}(\gamma)|\leq 1$. 
Since $U_{0,\gamma}^x=X_{\gamma}$ and the ultrafilter $p$ is free, for any $x\in X_{\gamma}$ there exists $m(x)<k(x)$ such that 
$$F_{x}=\{i\in\omega: z_i\in U_{m(x),\gamma}^x\setminus U_{m(x)+1,\gamma}^x\}\in p.$$
Note that the set $F_x$ actually depends on the pair $\big( U_{m(x),\gamma}^x, U_{m(x)+1,\gamma}^x\big)=E(x,m(x),\gamma)$.

Fix an arbitrary $f\in\Phi_{\gamma}$. By the compactness of $[0,1]$, the ultrafilter $f(p)$
%$$\{A\subseteq [0,1]: \exists P\in p \text{ such that } \bigcup_{i\in P}f(z_i)\subseteq A\}$$ 
on $[0,1]$ generated by the family $\{f(P): P\in p\}$ converges to a point
$b\in [0,1]$. That is for each $n\in\mathbb N$
$$P_{f,n}=\{i\in\w:f(z_i)\in(b-1/n,b+1/n)\}\in p.$$
Since $p$ is a $P$-point, there exists $G_f\in p$ such that $G_f\subset^*P_{f,1/n} $ for every $n\in\N$. It is easy to see that the sequence $\{f(z_i):i\in G_f\}$  converges to $b$. 

Since $|\mathcal B_\gamma\cup \Phi_{\gamma}|<\mathfrak{c}$ we obtain that 
$|\{F_x:x\in X_\gamma\}\cup\{G_f: f\in\Phi_{\gamma}\}|<\mathfrak c,$ and hence there exists $F\in p$ such that $F\subset^{*}F_x\cap G_f$ for any $x\in X_{\gamma}$ and  $f\in\Phi_\gamma$. 

For any $n\in\w$ and $x\in X_{\gamma}$ set
$$b_n^x=\lim_{i\in F}f_{E(x,n,\gamma)}(z_i)$$
and note that $b^x_n=0$ for $n>m(x)$.

Let 
$d_{\gamma}=\{z_i: i\in F\}.$ 
Condition (4) implies that $\cl_{Y_{\gamma}}(U^x_{n+1,\gamma})\subseteq U^x_{n,\gamma}$ for every $x\in X_\gamma$ and $n\in\w$.  
Then for any $x\in X_{\gamma}$ we have $$d_{\gamma}\subset^{*}U_{m(x),\gamma}^x\setminus U_{m(x)+1,\gamma}^x\subset U_{m(x),\gamma}^x\setminus \cl_{Y_{\gamma}}\big(U_{m(x)+2,\gamma}^x\big).$$ 
It follows that for every $x\in X_{\gamma}$ there exists a function $g_x\in \w^F$ such that for all but finitely many $i\in F$ the following holds:
\begin{equation*}
\cl_{Y_{\gamma}}(U^{z_i}_{g_x(i),\gamma})\subset U_{m(x),\gamma}^x\setminus  \cl_{Y_{\gamma}}\big(U_{m(x)+2,\gamma}^x\big).
\end{equation*}

By the continuity, for any $f\in\Phi_{\gamma}$ %and $x\in X_\gamma$, 
 there is a  function $\pi_f\in\w^F$ such that 
\begin{equation}
\tag{$\circ_0$} f(U_{\pi_f(i),\gamma}^{z_i})\subset \big(f(z_i)-\frac{1}{2^{i+1}},f(z_i)+\frac{1}{2^{i+1}}\big)\cap [0,1].
\end{equation}

Let $h_{\gamma}\in\omega^{F}$ be a function such that $h_{\gamma}\geq^* g_{x}$ and $h_{\gamma}\geq^*\pi_f$ for each $x\in X_{\gamma}$ and $f\in\Phi_\gamma$. Such a function $h_{\gamma}$ exists because $|\mathcal B_{\gamma}\cup\Phi_{\gamma}|<\mathfrak{b}=\mathfrak{c}$. Without loss of generality we can additionally assume that $U_{h_{\gamma}(i),\gamma}^{z_i}\cap U_{h_{\gamma}(j),\gamma}^{z_j}=\emptyset$ for each distinct $i,j\in \w$.

We are in a position now to define
the  open neighborhood base $\mathcal B^\gamma_\xi$ at the point $\gamma$ in the space $Y_{\xi}$:
put 
\begin{equation}
\tag{$\circ_1$}
U_{0,\xi}^{\gamma}=X_{\xi}\quad  \hbox{ and }\quad U_{n,\xi}^{\gamma}=\bigcup_{i\in F\setminus n}U^{z_i}_{h_{\gamma}(i)+n,\gamma}\cup\{\gamma\} 
\end{equation}
 for all $n\in \N$.

Next, we define  $\mathcal B_{\xi}^x$ for $x\in X_{\gamma}$. Let $U_{0,\xi}^x=X_{\xi}$ for every $x\in X_{\gamma}$. 
For each  $n\in\N$ let $U_{n,\xi}^x=U_{n, \gamma}^x$ if either $n\geq m(x)+1$ or $n=m(x)$ and $b_n^x=0$, otherwise, put $U_{n,\xi}^x=U_{n,\gamma}^x\cup\{\gamma\}$.
It is easy to check that the family $\mathcal B_{\xi}=\{U_{n,\xi}^x: x\in X_{\xi},n\in\w\}$ forms a base of a topology $\tau_{\xi}$, and for each $x\in X_{\xi}$ the family $\mathcal B_{\xi}^{x}=\{U_{n,\xi}^x:n\in\w\}$ forms an open neighborhood base at $x$ in $Y_{\xi}$. Clearly, the first equality in condition (5) is satisfied for every $\beta<\alpha\leq \xi$ and $x\in X_\beta$. Fix an arbitrary $x\in X_{\gamma}$. The choice of $g_x$ and the inequality $g_x\leq^*h_\gamma$ ensure the existence of large enough $n\in\w$ such that 
$U^x_{m(x)+2,\xi}\cap U_{n,\xi}^\gamma=\emptyset$. 
Thus $Y_\xi$ is Hausdorff. 
The next claim proves the second equality in condition (5) for $\alpha=\gamma$, and therefore also for any $\alpha<\xi$.
\begin{claim} \label{dense_in_new}
In the space $Y_\xi$, $U_{n,\gamma}^x$ is dense in $U_{n,\xi}^x$ for all $n\in\w$ and 
$x\in X_\gamma$. 
\end{claim}
\begin{proof}
It suffices to consider the case $U_{n,\xi}^x=U_{n,\gamma}^x\cup\{\gamma\}$. Thus either 
$n< m(x)$ or $n=m(x)$ and $b_n^x>0$. In both of these cases 
$z_i\in U_{n,\gamma}^x$ for almost all $i\in F$, and hence
$U_{n,\gamma}^x$ intersects all neighborhoods of $\gamma$ in $Y_\xi$.
\end{proof}

Similarly as before, for each $n\in \w$ and $x\in X_\xi$ put $E(x,n,\xi)=(U_{n,\xi}^x, U_{n+1,\xi}^x)$.
Now we shall define a family $\Phi_\xi=\{f_{E(x,n,\xi)}:n\in\w, x\in X_\xi\}$ such that conditions (3), (4) and (6) are satisfied when $\alpha$ is replaced with $\xi$.
%where for each $U_{n,\xi}^x\in\mathcal B_{\xi}$ the function $f_{U_{n,\xi}^x}: Y_{\xi}\rightarrow [0,1]$ is continuous and $f_{U_{n,\xi}^x}{\restriction}_{X_\xi\setminus U_{n,\xi}^x}\equiv 0$. 
Put $f_{E(x,0,\xi)}\equiv 1$ for all $x\in X_{\xi}$. For each $x\in X_{\gamma}$ and $n\in\N$ set 
 $$f_{E(x,n,\xi)}=f_{E(x,n,\gamma)}\cup\{(\gamma,b^x_n)\}.$$ 
Conditions $(\circ_0)$, $(\circ_1)$ together with the choice of $b^x_n$ as the limit of the sequence 
$\{f_{E(x,n,\gamma)}(z_i):i\in F\}$, ensure the continuity of functions
$f_{E(x,n,\xi)}$ for all $n\in\w$ and $x\in X_\gamma$.
Finally, we set  
$$
f_{E(\gamma,n,\xi)}(y)=
\begin{cases}f_{E(z_i,h_{\gamma}(i)+n, \gamma)}(y) &\mbox{if $y\in U^{z_n}_{h_\gamma(i)+n,\gamma}$ and $i\in F\setminus n$};\\
1&\mbox{if $y=\gamma$};\\
0&\mbox{otherwise}.
\end{cases}
$$
Since the family $\{U_{h_{\gamma}(i),\gamma}^{z_i}:i\in\w\}$ is locally finite in $Y_\gamma$, it 
remains to check that $f_{E(\gamma,n,\xi)}$ 
 is continuous at $\gamma$. For this it is enough to prove the first equality of condition (4) for $x=\gamma$ and $\xi=\alpha$. This follows from the following equation:
 \begin{eqnarray*}
f_{E(\gamma,n,\xi)}(U_{n+1,\xi}^\gamma)=\{1\}\cup\bigcup_{i\in F\setminus(n+1)} f_{E(\gamma,n,\xi)}
(U^{z_i}_{h_\gamma(i)+n+1,\gamma})=\\ =\{1\}\cup\bigcup_{i\in F\setminus(n+1)} f_{E(z_i,h_{\gamma}(i)+n,\gamma)}
(U^{z_i}_{h_\gamma(i)+n+1,\gamma})=\{1\}. 
\end{eqnarray*}

The fact that functions from $\Phi_\xi$ satisfy the first equation from (4) also for $x\in X_{\gamma}$
follows from (5) established in Claim~\ref{dense_in_new} and from the continuity of functions in $\Phi_\xi$. 
It is clear that any function from $\Phi_\xi$ satisfies the second equation of condition (4). Finally,
it is a tedious routine to check that the families $\mathcal B_{\xi}$ and $\Phi_{\xi}$ satisfy conditions (1), (2), (3), and (6), 
which concludes case 2.
\medskip

3) Let $X_{\xi}=\bigcup_{\alpha\in \xi}X_{\alpha}$. For each $x\in X_{\xi}$ let $\theta_x=\min\{\alpha:x\in X_{\alpha}\}$. For each $x\in X_{\xi}$ and $n\in\w$ put $$U_{n,\xi}^x=\bigcup_{\theta_x\leq\alpha<\xi}U_{n,\alpha}^x \quad\hbox{and}\quad f_{E(x,n,\xi)}=\bigcup_{\theta_x\leq \alpha<\xi}f_{E(x,n,\alpha)}.$$
It is straightforward to check that the family $\mathcal B_{\xi}=\{U_{n,\xi}^x: n\in\w, x\in X_{\xi}\}$ forms a base of a topology $\tau_{\xi}$ on $X_\xi$, and the families $\Phi_{\xi}=\{f_{E(x,n,\xi)}: n\in\w, x\in X_\xi\}$ and $\mathcal B_{\xi}$ satisfy conditions (1)--(6).
Thus, it remains to establish the continuity of functions in $\Phi_\xi$, which we do next.

Assume to the contrary  that there exists  $f_{E(x,n,\xi)}\in \Phi_{\xi}$ which is not continuous at a point $z\in X_{\xi}$. Then there exists a sequence $\{a_k\}_{k\in\w}$ such that $\lim_{k\in\w}a_k=z$ but 
$$\lim_{k\in\w}f_{E(x,n,\xi)}(a_k)\neq f_{E(x,n,\xi)}(z).$$ For each $k\in\w$ let $\mu_k\in\xi$ be such that $\{a_{k},x,z\}\subset X_{\mu_k}$. 
%Let $\mu\in \xi$ be any ordinal such that $\{x,z\}\in X_\mu$. 
Since $Y_{\xi}$ is first-countable, $X_0$ is a dense subset of $X_{\xi}$ and the functions $f_{E(x,n,\mu_k)}$, $k\in\w$ are continuous on $Y_{\mu_k}$, for each $k\in\w$ we can pick $b_k\in X_0$ such that 
$$\lim_{k\in\w}b_k=\lim_{k\in\w}a_k=z\qquad \hbox{ and }\qquad |f_{E(x,n,\mu_k)}(b_k)-f_{E(x,n,\mu_k)}(a_k)|\leq 1/k. $$

For each $k\in\w$, by condition (6), $f_{E(x,n,\xi)}{\restriction}_{ X_{\mu_k}}=f_{E(x,n,\mu_k)}$ implying that \begin{equation}\tag{$\bullet$}|f_{E(x,n,\xi)}(b_k)-f_{E(x,n,\xi)}(a_k)|=|f_{E(x,n,\mu_k)}(b_k)-f_{E(x,n,\mu_k)}(a_k)|\leq 1/k.
\end{equation}
Since the function $f_{E(x,n,\mu_0)}$ is continuous on $Y_{\mu_0}$ condition (6) implies the following: 
$$\lim_{k\in\w}f_{E(x,n,\xi)}(b_k)=\lim_{k\in\w}f_{E(x,n,\mu_0)}(b_k)=f_{E(x,n,\mu_0)}(\lim_{k\in\w}b_k)=f_{E(x,n,\mu_0)}(z)=f_{E(x,n,\xi)}(z).$$
But then 
\begin{eqnarray*}0\neq f_{E(x,n,\xi)}(z)-\lim_{k\in\w}f_{E(x,n,\xi)}(a_k)=\lim_{k\in\w}f_{E(x,n,\xi)}(b_k)-\lim_{k\in\w}f_{E(x,n,\xi)}(a_k)=\\
=\lim_{k\in\w}\big(f_{E(x,n,\xi)}(b_k)-f_{E(x,n,\xi)}(a_k)\big)=0,
\end{eqnarray*}
where the latter equality follows from equation $(\bullet)$.
The obtained contradiction ensures that all functions in $\Phi_{\xi}$ are continuous, and thus
concludes case 3.

By the construction, the space $Y_{\mathfrak{c}}$ is Tychonoff, first-countable and contains $X$ as a dense subspace. Let $A$ be any countable subset of $Y_{\mathfrak{c}}$. If the set $B=A\cap \mathfrak{c}$ is infinite, then consider $h(B)\in\mathfrak{c}$. By the construction, either $B$ has an accumulation point in $Y_{h(B)}$ or $h(B)$ is an accumulation point of $B$ in $Y_{h(B)+1}$. In both cases $B$ has an accumulation point in $Y_{\mathfrak c}$. If $A\subset^* X$, then either it has an accumulation point in $X$, or $A$ is closed and discrete in $X$. In the latter case either $A$ has an accumulation point in $Y_{h(A)}$ or $h(A)$ is an accumulation point of $A$ in $Y_{h(A)+1}$. Thus $Y_{\mathfrak c}$ is countably compact,
which completes our proof.
\end{proof}

\subsection*{Proof of Theorem~\ref{ps}}
We need to show that the following assertions are equivalent:
\begin{enumerate}
    \item $\mathfrak b=\mathfrak s=\mathfrak c$.
    \item Every first-countable zero-dimensional Hausdorff space of weight $<\mathfrak c$ embeds densely into a first-countable zero-dimensional pseudocompact space.
    \item Every first-countable zero-dimensional Hausdorff space of cardinality $<\mathfrak c$ embeds densely into a first-countable zero-dimensional pseudocompact space.
\end{enumerate}

\begin{proof}
The implication (1) $\Rightarrow$ (2) follows from Proposition~\ref{pseudocompact}.  Since the weight of an infinite first-countable space doesn't exceed its cardinality, implication (2) $\Rightarrow$ (3) holds. 

(3) $\Rightarrow$ (1). Corollary~\ref{pseudob} implies that $\mathfrak b=\mathfrak c$. Assuming that $\mathfrak s<\mathfrak c$, we get that the normal first-countable space $Y$ constructed in Proposition~\ref{De} has cardinality $\mathfrak s<\mathfrak c$ and, thus, is zero-dimensional. Then, by the assumption, $Y$ embeds densely into a first-countable zero-dimensional pseudocompact space.  But this contradicts Proposition~\ref{De}. Hence  $\mathfrak b=\mathfrak s=\mathfrak c$.    
\end{proof}

\subsection*{Proof of Theorem~\ref{PR}}
We need to show that a subspace $X$ of the Cantor space is a $\lambda$-set if and only if the Pixley-Roy hyperspace $\PR(X)$ embeds densely into a first-countable pseudocompact space.
 
\begin{proof}
The implication ($\Rightarrow$) follows from \cite[Theorem 4.2]{Bell}. 
%Bell showed that if $X$ is a $\lam bda$-subset of the Cantor space, then $\PR(X)$ embeds densely into a first-countable pseudocompact space. 

The implication ($\Leftarrow$) follows from Proposition~\ref{bpseudo}. \end{proof}

\section*{Acknowledgements} We would like to thank prof. Alan Dow for an idea of proving Proposition~\ref{De}.

\end{document}